\documentclass[a4paper,12pt,twoside]{article}
\usepackage{amsmath}
\usepackage[psamsfonts]{amssymb}
\usepackage[psamsfonts]{eucal}
\usepackage{amsthm}
\usepackage[utf8]{inputenc}
\usepackage{textcomp}
\usepackage{mathptmx}
\usepackage[scaled]{helvet}
\usepackage[colorlinks=true]{hyperref}
\hypersetup{urlcolor=blue, citecolor=red}

\theoremstyle{plain}
\newcounter{theorem}

\newtheorem{proposition}[theorem]{Proposition}
\newtheorem{lemma}[theorem]{Lemma}
\newtheorem{corol}[theorem]{Corollary}
\theoremstyle{definition}
\newcounter{remark}
\newtheorem{rmk}[remark]{Remark}
\newtheorem{rmks}[remark]{Remarks}
\newcommand{\field}[1]{\mathbb{#1}}
\newcommand{\RR}{\field{R}}
\def\Ad{\mathop{\rm Ad}\nolimits}
\def\ad{\mathop{\rm ad}\nolimits}
\def\orth{\mathop{\rm orth}\nolimits}
\def\vect#1{\overrightarrow{#1}}
\def\Isom{\mathop{\rm Isom}\nolimits}
\def\SO{\mathop{\rm SO}\nolimits}
\def\vert{\mathop{\rm vert}\nolimits}
\def\leaves{\mathop{\rm Leaves}\nolimits}
\def\d{\mathop{\rm d}\nolimits}

\begin{document}

\thispagestyle{plain}

\title{On Henri Poincar\'e's note \\
       \lq\lq Sur une forme nouvelle des \'equations de la M\'ecanique\rq\rq}

\author{Charles-Michel Marle}

\date{}

\maketitle
\tableofcontents

\begin{abstract}
We present in modern language the contents of the famous note published by Henri Poincar\'e in 1901 
\lq\lq Sur une forme nouvelle des \'equations de la M\'ecanique\rq\rq, in which he proves that, 
when a Lie algebra acts locally transitively on the configuration space of a Lagrangian mechanical system, 
the well known Euler-Lagrange equations are equivalent to a new system of differential equations
defined on the product of the configuration space with the Lie algebra. We write these equations, called the
\emph{Euler-Poincar\'e equations}, under an intrinsic form, without any reference to a particular system of 
local coordinates, and prove that they can be conveniently expressed in terms of the 
Legendre and momentum maps. We discuss the use of the Euler-Poincar\'e equation for reduction 
(a procedure sometimes called \emph{Lagrangian reduction} by modern authors), 
and compare this procedure with the well known Hamiltonian reduction procedure 
(formulated in modern terms in 1974 by J.E.~Marsden and A.~Weinstein). We explain how a break of symmetry in the phase space produces the appearance of a semi-direct product of groups.
\end{abstract}

\label{first}

{\it In memory of Jean-Marie Souriau, founder of the modern theory of Geometrical Mechanics, with respect and admiration}

\section{Introduction}\label{intro}
On the 19th of February 1901, Henri Poincar\'e published a short note~\cite{Poinca} entitled \lq\lq Sur une forme nouvelle des \'equations de la M\'ecanique\rq\rq\ in which he considers a Lagrangian mechanical system with a configuration space on which a Lie algebra acts locally transitively (it means that there exists on the configuration space a Lie algebra of vector fields such that, at each point, the values of these vector fields completely fill the tangent space). Poincar\'e proves that the equations of motion can be written as differential equations living on the product of the configuration space with the Lie algebra, rather than on the tangent bundle to the configuration space. Of course, these equations are equivalent to the well known Euler-Lagrange equations, as was shown by Poincar\'e himself in his note, but they are written in terms of different variables. More recently, several scientists working in the field called \lq\lq Geometric Mechanics\rq\rq, used the equations obtained by Poincar\'e (which they called \lq\lq Euler-Poincar\'e equations\rq\rq) to solve various problems. Following a remark made by Poincar\'e at the end of his note, several authors observed that these equations become very simple when the Lagrangian $L$ is such that its value $L(v)$ at a vector $v$ tangent to the configuration space at a point $x$, only depends on the element of the Lie algebra of vector fields which, at the point $x$, takes the value $v$. Modern authors sometimes call \lq\lq Lagrangian reduction\rq\rq~\cite{bloch, castrillon, cendra, cendraholm, cendramarsden, holm, 
holmmarsden, holm2, holm3, ratiu} the use of that property to make easier the determination of motions of the system.
\par\smallskip

Assumptions made in these recent papers and books seem to us very often more restrictive than those made by 
Poincar\'e himself; for example, several modern authors assume that the mechanical system under study has a 
Lie group as configuration space, and that its dynamics is described by a Lagrangian invariant under the 
lift to the tangent bundle of the action of this group on itself by translations either on the right or 
on the left. It seemed to us useful to go back to the original source, Poincar\'e's
note, to see whether some ideas of Poincar\'e were not overlooked by modern authors.
\par\smallskip

The contents of Poincar\'e's note are described in modern language in Section~\ref{PoincaNote}. 
The equation\footnote{%
We will write \lq\lq Equation\rq\rq\ rather than \lq\lq Equations\rq\rq\ for a differential equation, considered as a single mathematical object even when it lives on a manifold whose dimension may be larger than $1$, and therefore may appear, when written in local coordinates, as a system of several equations.} 
derived in this note, which will be called  the \emph{Euler-Poincar\'e equation}, is written both in local coordinates, as was done by Poincar\'e, and under an intrinsic geometric form with no reference to any particular system of coordinates. In Section~\ref{Examples}, following a remark made by Poincar\'e at the end of his note, we show that Euler-Lagrange equations and the Euler equations for the motion of a rigid body with a fixed point can be considered as special cases of the Euler-Poincar\'e equation. In
Section~\ref{LegMom} we prove that the Euler-Poincar\'e equation can be expressed in terms of the Legendre map and 
the momentum map of the lift to the cotangent bundle of the Lie algebra action on the configuration space. At the end of this Section we also discuss a simple example (the spherical pendulum) in which the dimension of the Lie algebra of vector fields is strictly larger that the dimension of the configuration space of the system. 
The procedure sometimes called \emph{Lagrangian reduction} by modern authors, which amounts, 
when the Lagrangian possesses some symmetry properties, to use the Euler-Poincar\'e equation in order to solve 
successively two differential equations defined on smaller dimensional spaces instead of a single differential equation on a higher dimensional space, is discussed in Section~\ref{EPR}. At the end of this Section, a simple example is used to show that when the dimension of the Lie algebra is strictly larger than the dimension of the configuration space, very serious obstructions limit the applicability of Lagrangian reduction. In Section~\ref{Ham} and in all that follows, the Lagrangian is assumed to be hyperregular and we discuss the Euler-Poincar\'e equation in Hamiltonian formalism. The reduction 
procedure in Hamiltonian formalism is more fully discussed in Section~\ref{EPRH}. 
In Section~\ref{confspaceliegroup} we assume that the configuration space of the system is a Lie group, and we fully discuss the Lagrangian reduction procedure and its relations with the better known Marsden-Weinstein reduction procedure. Finally in Section \ref{SymBreak} we explain how a break of symmetry in the cotangent bundle to the configuration space can lead to the appearance of an extended action of a semi-direct product of groups.

\section{Poincar\'e's Note}\label{PoincaNote}
\subsection{Derivation of the Euler-Poincar\'e equation.}
Poincar\'e considers a Lagrangian mechanical system whose configuration space is a smooth manifold $Q$. The Lagrangian
is a smooth real valued function $L$ defined on the tangent bundle $TQ$. To each parametrized continuous, piecewise smooth
curve $\gamma:[t_0,t_1]\to Q$, defined on a closed interval $[t_0,t_1]$, with values in $Q$, one associates the value at $\gamma$ of the action integral $I_L$
 $$I_L(\gamma)=\int_{t_0}^{t_1}L\left(\frac{\\d\gamma(t)}{\d t}\right)\,\d t\,.$$
The equation of motion of the Lagrangian system is obtained by writing that the pa\-ra\-me\-tri\-zed curve $\gamma$ is an extremal of 
$I$, for variations of $\gamma$ with fixed endpoints.
\par\smallskip

 Poincar\'e assumes that a finite dimensional Lie algebra $\frak g$ acts on the configuration manifold $Q$. In other words, he assumes that there exists a smooth Lie algebras homomorphism $\psi$ of $\frak g$ into the Lie algebra $A^1(Q)$ of smooth vector fields on $Q$. More exactly, Poincar\'e's considerations being local, he only assumes that for each point $a$ of $Q$, 
there exists a neighbourhood $U_a$ of that point and a smooth Lie algebras homomorphism $\psi_a$ of the Lie algebra $\frak g$ into the Lie algebra $A^1(U_a)$ of smooth vector fields on $U_a$. For simplicity we will assume that the Lie algebras homomorphism $\psi$ takes its value into the Lie algebra $A^1(Q)$ of smooth vector fields everywhere defined on $Q$. Since the purpose of Poincar\'e's note is to obtain local expressions of the equation of motion, the more general case when the homomorphisms $\psi_a$ take their values into the space of vector fields on open subsets $U_a$ of $Q$ is easily treated by replacing $Q$ by $U_a$. 
\par\smallskip

For each $X\in\frak g$, we will say that $\psi(X)$ is the \emph{fundamental vector field on $Q$ associated to $X$}. In order to shorten the notations, we will write
$X_Q$ for $\psi(X)$. 
\par\smallskip

Poincar\'e assumes that $\psi$ is locally transitive, \emph{i.e.}, that for each $x\in Q$, the set of values 
taken by $X_Q(x)$, for all $X\in{\frak g}$, is the whole tangent space $T_xQ$. In other words, Poincar\'e assumes that the vector bundles homomorphism, defined on the trivial vector bundle $Q\times{\frak g}$, with values in the tangent bundle $TQ$,
  $$\varphi(x,X)=X_Q(x)\,,
  \quad\hbox{with}\quad x\in Q\,,\ X\in{\frak g}\,,$$
is surjective. We will see that its transpose 
$\varphi^T:T^*Q\to Q\times {\frak g}^*$, which is an injective vector bundles homomorphism 
of the cotangent bundle $T^*Q$ into the trivial bundle $Q\times{\frak g}^*$,
where ${\frak g}^*$ is the dual space of the vector space $\frak g$, is closely related 
to the momentum map defined by J.-M. Souriau~\cite{souriau}.
\par\smallskip 

Poincar\'e's assumptions are satisfied, for example, when there exists a locally transitive action $\Phi:G\times Q\to Q$ 
on the manifold $Q$ of a Lie group $G$ whose Lie algebra is ${\frak g}$.
\par\smallskip

For a given parametrized continuous, piecewise smooth curve 
$\gamma:[t_0,t_1]\to Q$, any pa\-ra\-me\-tri\-zed 
piecewise continuous and 
smooth curve $\overline\gamma:[t_0,t_1]\to Q\times {\frak g}$ 
which, for each $t\in[t_0,t_1]$ at which $\gamma$ is smooth, satifies 
 $$\varphi\bigl(\overline\gamma(t)\bigr)
 =\frac{\\d\gamma(t)}{\d t}\eqno(1)$$
will be said to be a \emph{lift of $\gamma$ to $Q\times{\frak g}$}.
\par\smallskip

Let
  $$p_Q:Q\times{\frak g}\to Q\quad\hbox{and}\quad
    p_{\frak g}:Q\times{\frak g}\to{\frak g}
  $$
be the canonical projections of the product $Q\times{\frak g}$ onto its two factors. Obviously $(1)$ implies
  $$p_Q\circ \overline\gamma=\gamma\,.
  $$
Therefore any lift $\overline\gamma$ to $Q\times{\frak g}$ of a continuous, piecewise smooth parametrized curve 
$\gamma:[t_0,t_1]\to{\frak g}$ can be written
 $$\overline\gamma=(\gamma, V)\,,$$
where $V=p_{\frak g}\circ\overline\gamma:[t_0,t_1]\to{\frak g}$ 
is a piecewise continuous and smooth parametrized curve which satifies, 
for each $t\in[t_0,t_1]$
at which $\gamma$ is smooth,
 $$\frac{\d\gamma(t)}{\d t}=\bigl(V(t)\bigr)_Q\bigl(\gamma(t)\bigr)\,.\eqno(2)
 $$
Any parametrized continuous, piecewise differentiable curve
$\gamma:[t_0,t_1]\to Q$ always has a lift to $Q\times{\frak g}$.
But such a lift may not be unique. Let us set indeed
  $$r=\dim{\frak g}\,,\quad n=\dim Q\,,$$
and for each $x\in Q$, let
 $${\frak g}_x=\bigl\{X\in{\frak g}\, ; X_Q(x)=0\bigr\}$$
be the isotropy Lie algebra of $x$. When 
$r>n$, $\dim{\frak g}_x=r-n>0$. If 
a curve $\overline\gamma=(\gamma,V)$
is a lift of $\gamma$, any other curve 
$\overline\gamma\,'=(\gamma,V')$ such that $(V'-V)(t)\in{\frak g}_{\gamma(t)}$ for each  
$t\in[t_0,t_1]$ at which $\gamma$ is smooth is another lift of $\gamma$.
\par\smallskip

Conversely, a piecewise continuous and smooth
curve $\overline\gamma:[t_0,t_1]\to Q\times{\frak g}$ is a lift to
$Q\times{\frak g}$ of a parametrized continuous, piecewise smooth
curve $\gamma:[t_0,t_1]\to Q$ if and only if its first component $p_Q\circ\overline\gamma$ is almost everywhere equal to 
$\gamma$ and its second component $V=p_{\frak g}\circ\overline\gamma$ satisfies condition $(2)$ above.
A piecewise continuous and smooth curve $\overline\gamma=(\gamma,V):[t_0,t_1]\to Q\times{\frak g}$ 
whose first component $\gamma$ is continuous and whose second component $V$ satisfies condition $(2)$ above will be said to be \emph{admissible}. Any admissible curve $\overline\gamma=(\gamma,V)$ is a lift to $Q\times{\frak g}$ of its first component
$\gamma$.
\par\smallskip

Let $\overline L:Q\times {\frak g}\to {\mathbb R}$ be the function
 $$\overline L=L\circ\varphi:(x,X)\mapsto L\bigl(X_Q(x)\bigr)\,,
  \quad x\in Q\,,\ X\in{\frak g}\,,
 $$
and let $\overline I_{\overline L}$ be the functional, defined on the space of
of parametrized piecewise continuous curves 
$\overline\gamma:[t_0,t_1]\to Q\times{\frak g}$,
  $$\overline I_{\overline L}(\overline\gamma)=\int_{t_0}^{t_1} \overline L
  \circ\overline\gamma(t)\,\d t\,.
  $$
If $\gamma:[t_0,t_1]\to Q$ is a parametrized continuous, piecewise differentiable curve in $Q$, and 
$\overline\gamma:[t_0,t_1]\to Q\times{\frak g}$ 
any lift of $\gamma$ to $Q\times{\frak g}$, we have
 $$\overline I_{\overline L}(\overline\gamma)=I_L(\gamma))\,.$$
Therefore looking for continuous, piecewise smooth maps $\gamma:[t_0,t_1]\to Q$ at which $I$ is stationary amounts to look for admissible curves $\overline \gamma:[t_0,t_1]\to Q\times{\frak g}$ at which $\overline I_{\overline L}$
\emph{restricted to the space of admissible curves in} $Q\times{\frak g}$ is stationary. The  equation so obtained will be, of course, equivalent to the Euler-Lagrange equation, but will be expressed differently, with different variables.
\par\smallskip

In order to write that the parametrized continuous and piecewise smooth curve $\gamma: [t_0,t_1]\to Q$ is an extremal of $I$,
Poincar\'e considers a \emph{variation with fixed endpoints} of that curve, \emph{i.e.}, a continuous and piecewise smooth
map $(t,s)\mapsto \gamma_s(t)$, defined on the product of intervals
$[t_0,t_1]\times[-\varepsilon,\varepsilon]$, with values in $Q$, such that
 \begin{equation*}
 \begin{split}
   \gamma_0(t)&=\gamma(t)\quad\hbox{for each\ }t\in[t_0,t_1]\,,\\
   \gamma_s(t_0)&=\gamma(t_0)\quad\hbox{and}\quad\d\gamma_s(t_1)=\gamma(t_1) \quad\hbox{for each\ }s\in[-\varepsilon,\varepsilon]\,.
 \end{split}
 \end{equation*}  
There exists a (non unique) piecewise continuous and smooth map $(t,s)\mapsto \overline\gamma_s(t)$, defined 
on the product of intervals $[t_0,t_1]\times[-\varepsilon,\varepsilon]$, with values in $Q\times{\frak g}$,\
such that for each $s\in[-\varepsilon,\varepsilon]$, $\overline\gamma_s$ is a lift of $\gamma_s$ to $Q\times{\frak g}$.
In other words, the map $(t,s)\mapsto\overline\gamma(t,s)$ is such that for each $t\in[t_0,t_1]$ and each $s\in[-\varepsilon,\varepsilon]$,
    $$\overline\gamma_s(t)=\bigl(\gamma_s(t),V_s(t)\bigr)\,,\quad\hbox{with}\ V_s(t)\in{\frak g}\,,$$
and, for each $(t,s)$ at which the map $(t,s)\mapsto\gamma_s(t)$ is smooth
    $$\bigl(V_s(t)\bigr)_Q\bigl(\gamma_s(t)\bigr)=\frac{\d\gamma_s(t)}{\d t}\,.\eqno(3)$$
The parametrized curve $\gamma$ is an extremal of $I$ if and only if, for any variarion $(t,s)\mapsto\gamma_s(t)$ with fixed endpoints of $\gamma$, we have
 $$\frac{\d I_L(\gamma_s)}{\d s}\Bigl|_{s=0}=0\,.$$
Poincar\'e uses the fact that, for each $s\in[-\varepsilon,\varepsilon]$
 $$I_L(\gamma_s)=\overline I_{\overline L}(\overline\gamma_s)\,,\quad\hbox{which implies}\quad
 \frac{\d I_L(\gamma_s)}{\d s}\Bigl|_{s=0}=\frac{\d \overline I_{\overline L}(\overline\gamma_s)}{\d s}\Bigl|_{s=0}\,.$$
Therefore he can write
 $$\frac{\d I_L(\gamma_s)}{\d s}\Bigl|_{s=0} 
  = \frac{\d}{\d s}\left(\int_{t_0}^{t_1}\overline L\bigl(\gamma_s(t),V_s(t)\bigr)\,\d t\right)\Biggm|_{s=0}\,.$$
In local coordinates, the function ${\overline L}:Q\times{\frak g}\to {\mathbb R}$ is expressed as a functions of
$n+r$ real variables: the $n$ local coordinates $(x^1,\ldots,x^n)$ of $x\in Q$ (for a given chart of $Q$) 
and the $r$ components $(X^1, \ldots, X^r)$ of $X\in {\frak g}$ in a given basis $(X_1,\ldots,X_r)$ of $\frak g$.
Therefore
 \begin{equation*}
 \begin{split}
 \frac{\d I_L(\gamma_s)}{\d s}\Bigl|_{s=0}= \int_{t_0}^{t_1}&\Biggl[
  \sum_{i=1}^n\frac{\partial\overline L(\gamma_s(t), V_s(t))}{\partial x^i}\frac{\partial \gamma_s^i(t)}
{\partial s}\\
  +&\sum_{k=1}^r\frac{\partial\overline L(\gamma_s(t), V_s(t))}{\partial X^k}\frac{\partial V_s^k(t)}{\partial s}\Biggr]
  \Biggm|_{s=0}\,\d t\,.
 \end{split}
 \end{equation*}
We set
 $$\frac{\partial \gamma_s^i(t)}{\partial s}\Bigm|_{s=0}=\delta \gamma^i(t)\,,\quad \frac{\partial V_s^k(t)}{\partial s}\Bigm|_{s=0}=\delta V^k(t)\,,$$
so we may write
\begin{equation*} 
 \frac{\d I_L(\gamma_s)}{\d s}\Bigl|_{s=0}= \int_{t_0}^{t_1}\left[
  \sum_{i=1}^n\frac{\partial\overline L\bigl(\gamma(t), V(t)\bigr)}{\partial x^i}\delta \gamma^i(t)
  +\sum_{k=1}^r\frac{\partial\overline L\bigl(\gamma(t), V(t)\bigr)}{\partial X^k}\delta V^k(t)\right]\d t.
 \end{equation*}
For each $t$, the $\delta \gamma^i(t)$ are the component of a vector $\delta \gamma(t)\in T_{\gamma(t)}Q$
and the $\delta V^k(t)$ the components of a vector $\delta V(t)\in T_{V(t)}{\frak g}$,where
the vector space $T_{V(t)}{\frak g}$ tangent to $\frak g$ at $V(t)$ is, of course, 
canonically isomorphic to $\frak g$. Let 
 $$\d_Q\overline L:Q\times{\frak g}\to T^* Q\quad\hbox{and}\quad
    \d_{\frak g}\overline L:Q\times{\frak g}\to {\frak g}^*
 $$ 
be the partial differentials of the function 
$\overline L:Q\times{\frak g}\to {\mathbb R}$ with respect to its first and its second variable. The above equality
can be written more concisely as
  \begin{equation*} 
 \frac{\d I_L(\gamma_s)}{\d s}\Bigl|_{s=0}= \int_{t_0}^{t_1}\left[
  \Bigl\langle \d_Q\overline L\bigl(\gamma(t), V(t)\bigr),\delta \gamma(t)\Bigr\rangle
  +\Bigl\langle \d_{\frak g}\overline L\bigl(\gamma(t), V(t)\bigr),\delta V(t)\Bigr\rangle\right]\,\d t\,.
 \end{equation*}
Since $\delta\gamma(t)\in T_{\gamma(t)}Q$, there exists an 
element\footnote{Here our notations differ slightly from those of Poincar\'e, who writes
$\omega(t)$ where we write $\delta\omega(t)$. We used the symbol $\delta$ to indicate that 
$\delta\omega(t)$ is an infinitesimal quantity.} $\delta\omega(t)\in{\frak g}$ (non unique when $r>n$) such that 
 $$\delta \gamma(t)=\varphi\bigl(\gamma(t),\delta\omega(t)\bigr)
                   =\bigl(\delta\omega(t)\bigr)_Q\bigl(\gamma(t)\bigr)\,.$$
We may impose $\delta\omega(t_0)=0$ and $\delta\omega(t_1)=0$ since $\delta \gamma(t)$ vanishes for $t=t_0$ and $t=t_1$.
\par\smallskip

Replacing $\delta\gamma(t)$ by its expression in terms of $\delta\omega(t)$, we may write
 $$\Bigl\langle \d_Q\overline L\bigl(\gamma(t), V(t)\bigr),\delta \gamma(t)\Bigr\rangle
   =\Bigl\langle p_{{\frak g}^*}\circ\varphi^T\circ \d_Q\overline L
        \bigl(\gamma(t),V(t)\bigr),\delta\omega(t)\Bigr\rangle\,, 
 $$
where $p_{{\frak g}^*}:Q\times{\frak g}^*\to{\frak g}^*$ is the canonical projection on the second factor and
$\varphi^T:T^*Q\to Q\times{\frak g}^*$ the injective vector bundles homomorphism transpose of the surjective
vector bundle homomorphism $\varphi:Q\times{\frak g}\to TQ$. Poincar\'e denotes by $\Omega$ the map
 $$\Omega=p_{{\frak g}^*}\circ\varphi^T\circ \d_Q\overline L:Q\times{\frak g}\to{\frak g}^*\,.$$
The expressions of its components in the basis of ${\frak g}^*$ dual of the basis $(X_1,\ldots,X_r)$ of 
$\frak g$ are
 $$\Omega_k(x,X)=\sum_{i=1}^n\frac{\partial{\overline L}(x,X)}{\partial x^i}\bigl(X_k\bigr)_Q^i(x)\,.
 $$
The expression of the derivative of $I_L(\gamma_s)$ with respect to $s$, for $s=0$, becomes
\begin{equation*}
 \frac{\d I_L(\gamma_s)}{\d s}\Bigl|_{s=0}= \int_{t_0}^{t_1}\left(\Bigl\langle\Omega\bigl(\gamma(t),V(t)\bigr),\delta\omega(t)\Bigr\rangle
  +\Bigl\langle \d_{\frak g}\overline L\bigl(\gamma(t), V(t)\bigr),\delta V(t)\Bigr\rangle\right)\,\d t\,.
 \end{equation*}
Then Poincar\'e writes, without further explanation, \lq\lq Or on trouve ais\'ement 
 $$\delta V^i(t)=\frac{\d \bigl(\delta\omega^i(t)\bigr)}{\d t}+\sum_{(s,k)=(1,1)}^{(r,r)} c_{lk}^iV^k(t)\delta\omega^l(t)
\hbox{\ \rq\rq}\,,$$
$V^k(t)$ ($1\leq k\leq r$) and $\delta\omega^l(t)$ ($1\leq l\leq r$) being the components of $V(t)$ and $\delta\omega(t)$ in the basis $(X_1,\ldots,X_r)$ of the Lie algebra 
$\frak g$, the $c_{lk}^i$ ($1\leq i,k,l\leq r$) being the stucture constants of that Lie algebra in that basis.    
\par\smallskip

Poincar\'e probably obtained that result as follows. Let us calculate the derivatives with respect to $s$ of both sides of Equality $(3)$, and then set $s=0$. Since that equality is satisfied for all $s\in[-\varepsilon, \varepsilon]$, we must have
 $$\frac{\d}{\d s}\Bigl(\bigl(V_s(t)\bigr)_Q\bigl(\gamma_s(t)\bigr)\Bigr)\Bigm|_{s=0}
   =\frac{\d}{\d s}\left(\frac{\d\gamma_s(t)}{\d t}\right)\Bigm|_{s=0}\,.\eqno{(4)}
 $$
Both sides of the above equality are vectors tangent to $TQ$ at 
$\displaystyle\frac{\d\gamma(t)}{\d t}$. In order to evaluate these vectors, let 
$x^1,\ldots,x^n$ be the local coordinates in an admissible chart of $N$ whose domain
contains $\gamma(t)$, and $x^1,\ldots,x^n,\allowbreak v^1,\ldots,v^n$ be the local
coordinates in the associated chart of $TQ$. The local coordinates of 
$\gamma(t)$ can be written $\gamma^i(s,t)$ ($1\leq i\leq n$, where the $\gamma^i$
are smooth functions of the two real variables $s$ and $t$. The local coordinates of 
$\displaystyle\frac{\d\gamma_s(t)}{\d t}$ and 
$\displaystyle\frac{\d}{\d s}\bigl(\gamma_s(t)\bigr)\Bigm|_{s=0}$ 
are, respectively,
 $$\gamma^i(s,t)\,,\quad \frac{\partial\gamma^j(s,t)}{\partial t}\,,
                    \quad 1\leq i,j\leq n\,,
 $$
and
 $$\gamma^i(0,t)=\gamma^i(t)\,,\quad 
                  \frac{\partial\gamma^j(s,t)}{\partial s}\Bigm|_{s=0}
                  =\delta\gamma^j(t)\,,
                    \quad 1\leq i,j\leq n\,.
 $$
Let $x^i,v^j,\dot x^k,\dot v^l$ ($1\leq i,j,k,l\leq n$) be the local coordinates in the chart of $T(TQ)$ associated to the considered chart of $TQ$. The coordinates
$x^i$, $v^j$, $\dot x^k$ and $\dot v^l$ will be called, respectively, the first, second,
third and fourth set of $n$ coordinates of an element in $T(TQ)$. The
first set of $n$ coordinates of $\displaystyle \frac{\d}{\d s}\Bigl(\bigl(V_s(t)\bigr)_Q\bigl(\gamma_s(t)\bigr)\Bigr)\Bigm|_{s=0}$ is $\gamma^i(t)$ ($1\leq i\leq n$); the second is 
$\displaystyle\frac{\partial\gamma^j(s,t)}{\partial t}\Bigm|_{s=0}=\frac{\d\gamma^j(t)}{\d t}$ ($1\leq j\leq n$), and the third is $\delta\gamma^k(t)$ ($1\leq k\leq n$). The fourth is
 \begin{align*}
  \frac{\partial}{\partial s}\left(\sum_{j=1}^rV^j_s(t)\bigl(X_j\bigr)_Q^l
                                 \bigl(\gamma_s(t)\bigr)
                                  \right)\Biggm|_{s=0}
  &=\sum_{j=1}^r\delta V^j(t)\bigl(X_j\bigr)_Q^l
                                 \bigl(\gamma(t)\bigr)\\
  &\quad +\sum_{j=1}^rV^j(t)\sum_{i=1}^n\delta\gamma^i(t)\frac{\partial(X_j)_Q^l}
      {\partial x^i}\bigl(\gamma(t)\bigr)\,,\quad (1\leq l\leq n)\,,  
 \end{align*}
where $(X_j)_Q^l(x^1,\ldots,x^n)$ is the $n+l$-th coordinate, in the considered chart of 
$TQ$, of the value taken by the vector field $(X_j)_Q$ at the point of
$Q$ of coordinates $x^1,\ldots,x^n$, and where we have written $\gamma(t)$ for $\gamma^1(t),\ldots,\gamma^n(t)$. By using the equalities
$\delta\gamma(t)=\bigl(\delta\omega(t)\bigr)_Q\bigl(\gamma(t)\bigr)$, 
$\displaystyle\sum_{j=1}^r\delta V^j(t)(X_j)_Q=\bigl(\delta V)_Q$ and
$\displaystyle\sum_{j=1}^rV^j(t)(X_j)_Q=\bigl(V)_Q$,
we can write
  \begin{align*}
  \frac{\partial}{\partial s}
  &\left(\sum_{j=1}^rV^j_s(t)\bigl(X_j\bigr)_Q^l
                                 \bigl(\gamma_s(t)\bigr)
                                  \right)\Biggm|_{s=0}\\
  &=\sum_{j=1}^r\delta V^j(t)\bigl(X_j\bigr)_Q^l\bigl(\gamma(t)\bigr)
    +\sum_{j=1}^rV^j(t)\sum_{i=1}^n\bigl(\delta\omega\bigr)_Q^i
      \bigl(\gamma(t)\bigr)
       \frac{\partial(X_j)_Q^l}
      {\partial x^i}\bigl(\gamma(t)\bigr)\\
  &=\bigl(\delta V(t)\bigr)_Q^l\bigl(\gamma(t)\bigr)
    +\sum_{i=1}^n\bigl(\delta\omega\bigr)_Q^i\bigl(\gamma(t)\bigr)
      \frac{\partial}{\partial x^i}\left(
       \sum_{j=1}^rV^j(t)(X_j)_Q^l\right)
        \bigl(\gamma(t)\bigr)\\
  &=\bigl(\delta V(t)\bigr)_Q^l\bigl(\gamma(t)\bigr)
    +\sum_{i=1}^n\bigl(\delta\omega\bigr)_Q^i\bigl(\gamma(t)\bigr)
      \frac{\partial(V(t))_Q^l}{\partial x^i}
       \bigl(\gamma(t)\bigr)\,,\quad (1\leq l\leq n)\,, 
 \end{align*}
where $\bigl(\delta\omega\bigr)_Q^i(t)\bigl(\gamma(t)\bigr)$,
$\bigl(\delta V(t)\bigr)_Q^l\bigl(\gamma(t)\bigr)$ and
$\bigl(V(t)\bigr)_Q^l\bigl(\gamma(t)\bigr)$
are, respectively, the the
$n+i$-th coordinate and the $n+l$-th coordinates, in the considered chart of $TQ$, of the value taken by the vector fields $(\delta\omega)_Q$, $\bigl(\delta V(t)\bigr)_Q$ and
$\bigl(V(t)\bigr)_Q$ at the point $\gamma(t)\in Q$. 
\par\smallskip

Let us now evaluate the right hand side of equality $(4)$. That vector is tangent to $TQ$
at $\displaystyle\frac{\d\gamma(t)}{\d t}$, while the vector
$\displaystyle\frac{\d}{\d t}\left(\frac{\partial\gamma_s(t)}{\partial s}\Bigm|_{s=0}\right)=\frac{\d\bigl(\delta\gamma(t)\bigr)}{\d t}
$
is tangent to $TQ$ at $\delta\gamma(t)$. These two vectors are therefore not equal.
However, they are related by 
 $$\frac{\d}{\d s}\left(\frac{\d\gamma_s(t)}{\d t}\right)\Biggm|_{s=0}
  =\kappa_Q\left(\frac{\d\bigl(\delta\gamma(t)\bigr)}{\d t}\right)\,,
 $$
where $\kappa_Q:T(TQ)\to T(TQ)$ is the canonical involution of the second tangent bundle 
$T(TQ)$ \cite{Tulczyjew1989}. It means that in the considered chart of $T(TQ)$
the first set of $n$ coordinates, as well as the fourth
set of $n$ coordinates, of these two vectors, are equal, 
while the second set of $n$ coordinates of each one of these two vectors 
is equal to the third set of $n$ coordinates of the other.
Therefore, the fourth set of $n$ coordinates of
$\displaystyle \frac{\d}{\d s}\left(\frac{\d\gamma_s(t)}{\d t}\right)\Biggm|_{s=0}$, being equal to the fourth set of $n$ coordinates of
$\displaystyle\frac{\d\bigl(\delta\gamma(t)\bigr)}{\d t}$, can be expressed as
 \begin{align*}
  \frac{\d\bigl(\delta\gamma^l(t)\bigr)}{\d t} 
   &=\frac{\d}{\d t}\Bigl(\bigl(\delta\omega(t)\bigr)_Q^l\bigl(\gamma(t)\bigr)\Bigr)
   =\frac{\d}{\d t}\left(\sum_{j=1}^r\delta\omega^j(t)(X_j)_Q^l\bigl(\gamma(t)\bigr)
     \right)\\
   &=\sum_{j=1}^r\left(\frac{\d\bigl(\delta\omega^j(t)\bigr)}{\d t}(X_j)_Q^l
      \bigl(\gamma(t)\bigr)+\delta\omega^j(t)\sum_{i=1}^n\frac{\d\gamma^i(t)}{\d t}
       \frac{\partial(X_j)_Q^l}{\partial x^i}\bigl(\gamma(t)\bigr)\right)\\
   &=\left(\frac{\d\bigl(\delta\omega(t)\bigr)}{\d t}\right)_Q^l\bigl(\gamma(t)\bigr)
     +\sum_{i=1}^n\frac{\d\gamma^i(t)}{\d t}\frac{\partial}{\partial x^i}
       \left(\sum_{j=1}^r\delta\omega^j(t)(X_j)_Q^l\right)\bigl(\gamma(t)\bigr)\\
   &=\left(\frac{\d\bigl(\delta\omega(t)\bigr)}{\d t}\right)_Q^l\bigl(\gamma(t)\bigr)
     +\sum_{i=1}^n\frac{\d\gamma^i(t)}{\d t}\frac{\partial
                   \bigl(\delta\omega(t)\bigr)_Q^l}{\partial x^i}
       \bigl(\gamma(t)\bigr)\\
   &=\left(\frac{\d\bigl(\delta\omega(t)\bigr)}{\d t}\right)_Q^l\bigl(\gamma(t)\bigr)
     +\sum_{i=1}^n\bigl(V(t)\bigr)_Q^i\frac{\partial
                   \bigl(\delta\omega(t)\bigr)_Q^l}{\partial x^i}
       \bigl(\gamma(t)\bigr)\,,
        \quad (1\leq l\leq n)\,,
 \end{align*}
where we have used the equality 
$\displaystyle\frac{\d\gamma(t)}{\d t}=\bigl(V(t)\bigr)_Q\bigl(\gamma(t)\bigr)$.
\par\smallskip

Equality $(4)$ therefore leads, for each $l$ ($1\leq l\leq n$) to
 \begin{align*}
  \bigl(\delta V(t)\bigr)_Q^l\bigl(\gamma(t)\bigr)
    &+\sum_{i=1}^n\bigl(\delta\omega\bigr)_Q^i\bigl(\gamma(t)\bigr)
      \frac{\partial(V(t))_Q^l}{\partial x^i}
       \bigl(\gamma(t)\bigr)\\
   &=\left(\frac{\d\bigl(\delta\omega(t)\bigr)}{\d t}\right)_Q^l\bigl(\gamma(t)\bigr)
     +\sum_{i=1}^n\bigl(V(t)\bigr)_Q^j\frac{\partial
                   \bigl(\delta\omega(t)\bigr)_Q^l}{\partial x^i}
       \bigl(\gamma(t)\bigr)\,.
 \end{align*}
After reordering, we see that the right hand side is the $n+l$-th coordinate 
of the value at $\gamma(t)$ of the Lie bracket of the vector fields 
$\bigl(V(t)\bigr)_Q$ and $\bigl(\delta\omega(t)\bigr)_Q$. Moreover, 
since the map, which associates to each $Y\in{\frak g}$, the vector field $Y_Q$,
is a Lie algebras homomorphism, we have 
$\Bigl[\bigl(V(t)\bigr)_Q,\bigl(\delta\omega(t)\bigr)_Q\Bigr]
=\bigl[V(t), \delta\omega(t)\bigr]_Q$. So we have
 \begin{align*}
  \left(\delta V(t)-\frac{\d\bigl(\delta\omega(t)\bigr)}{\d t}
   \right)_Q^l
   \bigl(\gamma(t)\bigr)
  =
  &\sum_{j=1}^n\bigl(V(t)\bigr)_Q^j\bigl(\gamma(t)\bigr)\frac{\partial
                   \bigl(\delta\omega(t)\bigr)_Q^l}{\partial x^j}
         \bigl(\gamma(t)\bigr)\\
  &\quad-\sum_{i=1}^n\bigl(\delta\omega(t)\bigr)_Q^i\bigl(\gamma(t)\bigr)
      \frac{\partial(V(t))_Q^l}{\partial x^i}
       \bigl(\gamma(t)\bigr)\\
  &=\Bigl[\bigl(V(t)\bigr)_Q,\bigl(\delta\omega(t)\bigr)_Q\Bigr]^l
     \bigl(\gamma(t)\bigr)\\
  &=\bigl[V(t),\delta\omega(t)\bigr]_Q^l\bigl(\gamma(t)\bigr)\,,
     \quad(1\leq l\leq n)\,.
 \end{align*}
Therefore, for each $t\in[t_0,t_1]$,
 $$\delta V(t)-\frac{\d \bigl(\delta\omega(t)\bigr)}{\d t}
    -\bigl[V(t),\delta\omega(t)\bigr]\in{\frak g}_{\gamma(t)}\,. 
 $$
Since $\delta V(t)$ is determined only up to addition of a map $[t_0,t_1]\to{\frak g}$ which, for each
$t\in[t_0,t_1]$, takes it value in the isotropy Lie algebra ${\frak g}_{\gamma(t)}$, we can choose 
 $$\delta V(t)=\frac{\d \bigl(\delta\omega(t)\bigr)}{\d t}
    +\bigl[V(t),\delta\omega(t)\bigr]\,,$$
which is the result written in local coordinates by Poincar\'e.
Replacing $\delta V(t)$ by its expression, we obtain
\begin{equation*}
 \begin{split}
 \frac{\d I_L(\gamma_s)}{\d s}\Bigl|_{s=0}= \int_{t_0}^{t_1}&\Bigl[\Bigl\langle\Omega\bigl(\gamma(t),V(t)\bigr),\delta\omega(t)\Bigr\rangle\\
  &+\Bigl\langle \d_{\frak g}\overline L\bigl(\gamma(t), V(t)\bigr),\frac{\d \bigl(\delta\omega(t)\bigr)}{\d t}
    +\bigl[V(t),\delta\omega(t)\bigr]\Bigr\rangle\Bigr]\,\d t\,.
  \end{split} 
\end{equation*}
We transform the second term of the right hand side by writing
 \begin{equation*}
 \begin{split}
 \Bigl\langle \d_{\frak g}\overline L\bigl(\gamma(t), V(t)\bigr),\frac{\d \bigl(\delta\omega(t)\bigr)}{\d t}\Bigr\rangle
   &=\frac{\d }{\d t}\Bigl\langle \d_{\frak g}\overline L\bigl(\gamma(t), V(t)\bigr),\delta\omega(t)\Bigr\rangle\\
   & -\Bigl\langle \frac{\d }{\d t}\Bigl(\d_{\frak g}\overline L\bigl(\gamma(t), V(t)\bigr)\Bigr),\delta\omega(t)\Bigr\rangle\,,
 \end{split}
 \end{equation*}
therefore by integration
 \begin{equation*}
 \begin{split}\int_{t_0}^{t_1}
 \Bigl\langle \d_{\frak g}\overline L\bigl(\gamma(t), V(t)\bigr),\frac{\d \bigl(\delta\omega(t)\bigr)}{\d t}\Bigr\rangle
\d t
   &=\Bigl\langle \d_{\frak g}\overline L\bigl(\gamma(t), V(t)\bigr),\delta\omega(t)\Bigr\rangle\Bigm|_{t=t_0}^{t=t_1}\\
   &\phantom{=}-\int_{t_0}^{t_1}\Bigl\langle \frac{\d }{\d t}\Bigl(\d_{\frak g}\overline L\bigl(\gamma(t), V(t)\bigr)\Bigr),\delta\omega(t)\Bigr\rangle\d t\\
   &=-\int_{t_0}^{t_1}\Bigl\langle \frac{\d }{\d t}\Bigl(\d_{\frak g}\overline L\bigl(\gamma(t), V(t)\bigr)\Bigr),\delta\omega(t)\Bigr\rangle\d t 
 \end{split}
 \end{equation*}
since $\delta\omega(t_0)=\delta\omega(t_1)=0$. Similarly
 \begin{equation*}
  \begin{split}
   \Bigl\langle \d_{\frak g}\overline L\bigl(\gamma(t), V(t)\bigr),
    \bigl[V(t),\delta\omega(t)\bigr]\Bigr\rangle
    &=\Bigl\langle \d_{\frak g}\overline L\bigl(\gamma(t), V(t)\bigr),\ad_{V(t)}\bigl(\delta\omega(t)\bigr)\Bigr\rangle\\
    &=\Bigl\langle\ad^*_{V(t)}\Bigl(\d_{\frak g}\overline L\bigl(\gamma(t), V(t)\bigr)\Bigr),\delta\omega(t)\Bigr\rangle\,.
  \end{split}
 \end{equation*}
For each $V\in{\frak g}$ we have denoted by $\ad_V:{\frak g}\to{\frak g}$ the Lie algebras homomorphism
 $$\ad_V(X)=[V,X]=-[X,V]\,,$$ 
and by $\ad^*_V:{\frak g}^*\to{\frak g}^*$ the transpose of $\ad_V$, so that
 $$\langle\xi, \ad_V X\rangle=\langle\ad^*_V\xi,X\rangle\,,\quad \xi\in{\frak g}^*\,,\ V\ \hbox{and}\ X\in{\frak g}\,.$$
Finally we obtain
\begin{align*}
 \frac{\d I_L(\gamma_s)}{\d s}\Bigl|_{s=0}
 = \int_{t_0}^{t_1}&\Biggl(\Bigl\langle\Omega\bigl(\gamma(t),V(t)\bigr)\\
  -&\left(\frac{\d }{\d t}-\ad^*_{V(t)}\right)\Bigl(\d_{\frak g}\overline L\bigl(\gamma(t), V(t)\bigr)\Bigr)
  \,,\,\delta\omega(t)\Bigr\rangle\Biggr)\,\d t\,.
\end{align*}
Since $\delta\omega(t)$ can be chosen arbitrarily with the only restriction of vanishing at the end points, $\gamma$ is an extremal of $I$ if and only if
 $$\left(\frac{\d }{\d t}-\ad^*_{V(t)}\right)\Bigl(\d_{\frak g}\overline L\bigl(\gamma(t), V(t)\bigr)\Bigr)=\Omega\bigl(\gamma(t),V(t)\bigr)\,,\eqno(\hbox{\bf E-P1})
 $$
with
 $$\Omega=p_{{\frak g}^*}\circ\varphi^T\circ \d_Q\overline L\,.$$
It is the intrinsic expression (independent of any choice of local coordinates) of  
the \emph{Euler-Poincar\'e equation}. In his note,
Poincar\'e writes it, in local coordinates, under the form
\begin{equation*}
 \begin{split}
 \frac{\d }{\d t}\left(\frac{\partial{\overline L}\bigl(\gamma(t),V(t)\bigr)}{\partial X^i}\right)=&
 \Omega_i\bigl(\gamma(t),V(t)\bigr)\\
 &+\sum_{(k,s)=(1,1)}^{(r,r)} c^k_{is}V^s(t)\frac{\partial{\overline L}\bigl(\gamma(t),V(t)\bigr)}{\partial X^k}\,.
 \end{split}
 \end{equation*}
Of course, together with the Euler-Poincar\'e equation, we must consider the compatibility condition
 $$\frac{\\d\gamma(t)}{\d t}=\varphi\bigl(\gamma(t),V(t))\,.\eqno(\hbox{\bf CC})$$

\subsection{Comments made by Poincar\'e.}
At the end of his note, Poincar\'e briefly indicates that the $\Omega_i\bigl(\gamma(t),V(t)\bigr)$ 
can be interpreted as the components of forces exerted on the system. About his equation, 
which in local coordinates appears as a system of several equations, 
he indicates that they contain, as special cases, the well known Euler-Lagrange equations 
and the Euler equations governing the motion of a rigid body. Finally he writes
\lq\lq Elles sont surtout int\'eressantes dans le cas o\`u $U$ \'etant nul, $T$ ne d\'epend que des $\eta$\rq\rq.
He denoted by $T$ the kinetic energy expressed as a function defined on $Q\times{\frak g}$ rather than
on $TQ$, and by $U$ the potential energy, defined on $Q$. The function denoted by $T-U$ by Poincar\'e is therefore
$\overline L$ in our notations, and the variable $\eta$ on which $T$ depends is, in our notations,
the second variable $X$ on which depends $\overline L$. We see therefore that Poincar\'e writes 
that his equation is useful mainly when $\overline L:Q\times{\frak g}\to {\mathbb R}$ only depends on its second variable
$X\in{\frak g}$. This last remark made by Poincar\'e is the origin of the procedure called \emph{Lagrangian reduction} by modern authors, discussed in Section \ref{EPR}.

\section{Two Special Cases of the Euler-Poincar\'e Equation}\label{Examples}

\subsection{Euler-Lagrange equation.}
In the domain of a chart with local coordinates $x^1,\ldots,\allowbreak
x^n$, the configuration space $Q$ 
can be identified with an open subset of ${\mathbb R}^n$. The Lie algebra $\frak g$ is the Abelian Lie algebra ${\mathbb R}^n$, 
coordinates $X^1,\ldots,X^n$, with the zero bracket. The Lie algebras homomorphism $\psi$ is the linear map
$$\psi(X_i)=\frac{\partial}{\partial x^i}\,,\quad 1\leq i\leq n\,.$$
The vector bundle isomorphism $\varphi:Q\times{\frak g}\to TQ$ is given by
 $$\varphi(x,X_i)=\left(\frac{\partial}{\partial x^i}\right)(x)\,,\quad 1\leq i\leq n\,.$$
Let $L:TQ\to{\mathbb R}$ be the Lagrangian. In local coordinates,  the expression of $\overline L=L\circ\varphi$ 
is the same as that of $L$: 
 $${\overline L}(x^1,\ldots, x^n, X^1,\ldots, X^n)=L(x^1,\ldots, x^n, X^1,\ldots, X^n)\,.$$ 
Since for each $x\in{\mathbb R}^n$ $\varphi_x:{\frak g}\to T_xQ$ is expressed as the identity map, its transpose
$\varphi^T_x:T^*_xQ\to{\frak g}^*$ too is expressed as the identity map. The coadjoint action 
$\ad^*$ is identically zero since the Lie algebra $\frak g$ is Abelian. The Euler-Poincar\'e equation 
becomes 
 $$\frac{\d }{\d t}\Bigl(\d_{\frak g} L\bigl(\gamma(t),V(t)\bigr)\Bigr)=\d_Q L\bigl(\gamma(t),V(t)\bigr)\,.$$
We recognize the well known Euler-Lagrange equation.

\subsection{Euler Equation for the Motion of a Rigid Body.}\label{rigidbody}
In this section the reference frame considered is that in which the Earth is at rest. As a first approximation
we consider it as Galilean, the centrifugal force due to the Earth rotation exerted on a material body  
being included in its weight (the gravity force exerted by the Earth on that body) and the Coriolis force
being neglected. We study in that reference frame the motion of a material rigid body with at least 
three distinct non collinear material points.  We assume that units of time and of length, 
an origin of time and an orientation of space have been chosen. The physical space and the physical 
time can then be mathematically represented by an Euclidean three dimensional oriented affine space $E$
and by the real line ${\mathbb R}$, respectively.
A configuration of the body in space is represented by an affine, isometric, orientation preserving map, 
defined on an abstract Euclidean three dimensional oriented affine space $S$ 
(called the \emph{space of material points}), with values in $E$.
For each $z\in S$ representing some material point of the body, 
the position of that material point in space, when the configuration of the body 
in space is represented a map $x:S\to E$, is $x(z)$.
\par\smallskip
 
We assume, for simplicity, that a given point $O_S$ of the material 
body is constrained, by an ideal constraint, 
to remain at a fixed position $O_E$ in physical space. 
By choosing $O_S$ and $O_E$ as origins, respectively of $S$ and $E$, 
we can consider these spaces as \emph{vector} spaces. The set $Q$ of all possible positions of the material body 
in space is therefore the set  $\Isom(S,E)$ of linear, orientation preserving isometries of $S$ onto $E$.  
\par\smallskip 
 
Let $G_S$ and $G_E$ be the Lie groups (both isomorphic to 
${\rm SO}(3)$) of linear automorphisms of the oriented Euclidean vector spaces $S$ and $E$, respectively, 
${\frak g}_S$ and ${\frak g}_E$ their Lie algebras. The groups $G_S$ and $G_E$
both act on $Q$, respectively on the right and on the left, by two commuting, transitive and free actions 
$\Phi_S$ and $\Phi_E$, given by the formulae, where $x\in Q=\Isom(S,E)$, $g_S\in G_S$, $g_E\in G_E$,
 $$\Phi_S(x, g_S)=x\circ g_S\,,\quad \Phi_E(g_E,x)=g_E\circ x\,.
 $$ 
\par\smallskip

The values at $x\in Q$ of the fundamental vector fields on $Q$ associated to $X^S\in{\frak g}_S$ and $Y^E\in{\frak g}_E$ are 
 $$X^S_Q(x)=\frac{\d \bigl(x\circ\exp(sX^S)\bigr)}{ds}\Bigm|_{s=0}\,,\quad Y^E_Q(x)=\frac{\d (\exp(sY^E)\circ x)}{ds}\Bigm|_{s=0}\,.$$  
The maps $\psi_S:{\frak g}_S\to A^1(Q)$, $X^S\mapsto X^S_Q$, and $\psi_E:{\frak g}_E\to A^1(Q)$,
$Y^E\mapsto Y^E_Q$, are Lie algebras homomorphisms. However, one should be careful with signs: since $\Phi_S$ is an action of $G_S$ on the right, the bracket of elements in the Lie algebra ${\frak g}_S$ must be the bracket of 
\emph{left-invariant} vector fields on the Lie group $G_S$; similarly, since $\Phi_E$ is an action 
of $G_E$ on the left, the bracket of elements in the Lie algebra ${\frak g}_E$ must be the bracket 
of \emph{right-invariant} vector fields on the Lie group $G_E$. 
\par\smallskip

The maps $\varphi_S:Q\times{\frak g}_S\to TQ$ and $\varphi_E:Q\times{\frak g}_E\to TQ$, defined by
 $$\varphi_S(x,X^S)=X^S_Q(x)\,,\quad \varphi_E(x,Y^E)=Y^E_Q(x)\,,
    \quad x\in Q\,,\ X^S\in{\frak g}_S\,,\ Y^E\in{\frak g}_E$$
are vector bundles isomorphisms.
\par\smallskip

A \emph{motion} of the rigid body during a time interval $[t_0,t_1]$ is mathematically described by a smooth 
parametrized curve $\gamma:[t_0,t_1]\to Q$. For each time $t\in [t_0,t_1]$, there exists a unique 
$\Omega_S(t)\in{\frak g}_S$ and a unique $\Omega_E(t)\in{\frak g}_E$ such that
 $$\varphi_S\bigl(\gamma(t),\Omega_S(t)\bigr)=\frac{\\d\gamma(t)}{\d t}\,,\quad
   \varphi_E\bigl(\gamma(t),\Omega_E(t)\bigr)=\frac{\\d\gamma(t)}{\d t}\,.
 $$
In his beautiful paper~\cite{arnold}, Vladimir Arnold clearly explained their physical interpretation:  
$\displaystyle \frac{\\d\gamma(t)}{\d t}\in T_{\gamma(t)}Q$ is the \emph{true angular velocity} of the body, 
$\Omega_S(t)$ is the \emph{angular velocity of the body seen by an observer bound to the moving body} and moving with it,
and $\Omega_E(t)$ is the \emph{angular velocity of the body seen by an observer bound to the Galilean reference frame in which the motion is studied} and at rest with respect to that reference frame.
\par\smallskip

The \emph{kinetic energy} of the body is
 $$T\left(\frac{\d\gamma(t)}{\d t}\right)=\frac{1}{2} I\bigl(\Omega_S(t),\Omega_S(t)\bigr)\,,
 $$
where $I:{\frak g}_S\times{\frak g}_S\to{\mathbb R}$ is a symmetric, positive definite bilinear form
which describes the \emph{inertia properties} of the body. It does not depend on time nor on the configuration 
$\gamma(t)$ of the body.  We denote by $I^\flat:{\frak g}_S\to{\frak g}_S^*$
the linear map
 $$\bigl\langle I^\flat(X^S),Y^S\bigr\rangle=I(X^S,Y^S)\,,\quad X^S\ \hbox{and}\ Y^S\in{\frak g}_S\,.$$
\par\smallskip

The \emph{potential enegy} of the body, when its configuration is $x\in Q$, is
 $$U(x)=-\bigl\langle P,x(a)\bigr\rangle=\bigl\langle x^t(P),a\bigr\rangle\,,$$
where $a\in S$ is the vector whose origin is the fixed point $O_S$ and extremity the center of mass of the body,
and $P\in E^*$ is the gravity force. We will identify $E$ with its dual $E^*$ by using the Euclidean scalar product as pairing. Therefore $P$ can be seen as a vertical vector in $E$ directed downwards, equal to the weight of the body (product of its mass with the gravity acceleration). We have denoted by $x^t:E^*\to S^*$ the transpose of the isometry $x:S\to E$.
\par\smallskip

The Lagrangian $L$ is
 $$L\left(\frac{\d\gamma(t)}{\d t}\right)
    = \frac{1}{2}\,\Bigl\langle I^\flat\bigl(\Omega_S(t)\bigr),\Omega_S(t)\Bigr\rangle 
    -\bigl\langle \bigl(\gamma(t)\bigr)^t(P),a\bigr\rangle\,.$$
We use the vector bundle isomorphism $\varphi_S:Q\times{\frak g}_S\to TQ$ to derive the Euler-Poincar\'e equation.
With $\overline L=L\circ\varphi_S$, we have
 $${\overline L}(x,X^S)=\frac{1}{2}\,\bigl\langle I^\flat(X^S),X^S\bigr\rangle 
  -\bigl\langle x^t(P),a\bigr\rangle\,,\quad X^S\in{\frak g}_S\,,\ x\in Q\,.
 $$
The partial differentials of ${\overline L}$ are 
 $$\d_Q{\overline L}(x, X^S)=dU(x)\,,\quad \d_{\frak g}{\overline L}(x, X^S)=I^\flat(X^S)\,.$$
Therefore, the Euler-Poincar\'e equation is
 $$\frac{\d }{\d t}\Bigl(I^\flat\bigl(\Omega_S(t)\bigr)\Bigr)=-\ad^*_{\Omega_S(t)}\Bigl(I^\flat\bigl(\Omega_S(t)\bigr)\Bigr)
    +{^t}\!\varphi_{S}\Bigl(dU\bigl(\gamma(t)\bigr)\Bigr)\,.
 $$
We recognize the \emph{Euler equation} for the motion of a rigid body with a fixed point.

\section{The Euler-Poincar\'e Equation in Terms of the Legendre and the Momentum Maps}\label{LegMom}

\subsection{The Lift to $T^*Q$ of the Lie Algebra Action $\psi$ and the Momentum Map.}
Let us recall that the cotangent bundle $T^*Q$ of the configuration space, 
called the \emph{phase space} of our mechanical system, is endowed with a natural $1$-form $\eta$
called the \emph{Liouville form}, defined by
 $$\bigl\langle\eta(\xi), w\bigr\rangle=\bigl\langle\xi, T\pi_Q(w)\bigr\rangle\,,\quad \xi\in T^*Q\,,\ w\in T_\xi(T^*Q)\,,$$
where $\pi_Q:T^*Q\to Q$ is the canonical projection and $T\pi_Q:T(T^*Q)\to TQ$ its prolongation to vectors.
The \emph{canonical symplectic form} on $T^*Q$ is its exterior differential $\omega=d\eta$.
To each smooth function $f:T^*Q\to{\mathbb R}$ we can associate the vector field ${\mathcal X}_f$, called the
\emph{Hamiltonian vector field with Hamiltonian} $f$, defined by
 $$i({\mathcal X}_f)\omega=-df\,.$$

The action $\psi$ of the Lie algebra $\frak g$ on the configuration space $Q$ 
can be lifted, in a very natural way, into an action $\widehat\psi$ of $\frak g$ on the
cotangent bundle $T^*Q$ as follows (see for example~\cite{lima} chapter IV, proposition 1.19). For each $X\in{\frak g}$,  
the corresponding fundamental vector field on $Q$, $\psi(Q)=X_Q$, can be considered 
as a smooth real-valued function $f_{X_Q}$ on $T^*Q$, if we set
 $$f_{X_Q}(\xi)=\bigl\langle\xi, X_Q\circ \pi_Q(\xi)\bigr\rangle\,,\quad\xi\in T^*Q\,.$$
We can therefore take its associated Hamiltonian vector field ${\mathcal X}_{f_{X_Q}}$. We define the 
fundamental vector field on $T^*Q$ associated to $X$, for the lifted action $\widehat \psi$, as
 $$\widehat\psi(X)={\mathcal X}_{f_{X_Q}}\,.$$
To shorten the notations we will write $X_{T^*Q}$ for $\widehat \psi(X)$.
\par\smallskip

The action $\widehat \psi$ is Hamiltonian and admits the momentum map $J:T^*Q\to{\frak g}^*$
defined by
 $$\bigl\langle J(\xi),X\bigr\rangle=f_{X_Q}(\xi)=\bigl\langle\xi, X_Q\circ \pi_Q(\xi)\bigr\rangle\,,
   \quad\xi\in T^*Q\,,\ X\in{\frak g}\,.$$
When the Lie algebra action $\psi$ comes from an action $\Psi$ of a Lie group $G$, the momentum map $J$ is
said to be $\Ad^*$-equivariant, which means that it is equivariant
with respect to the action of $G$ on $T^*Q$ lifted from $\Psi$, and the coadjoint action of $G$ on 
the dual ${\frak g}^*$ of its Lie algebra.

Observing that $X_Q\circ\pi_Q(\xi)=\varphi\bigl(\pi_Q(\xi),X\bigr)$ and using the transpose 
$\varphi^T: T^*Q\to Q\times{\frak g}^*$ of the vector bundle isomorphism $\varphi$, we see that
 $$J=p_{{\frak g}^*}\circ\varphi^T\,,\quad\hbox{in other words}\quad \varphi^T=(\pi_Q,J)\,.$$   

\subsection{The Legendre Map $\mathcal L$.}
The \emph{vertical differential} 
of a smooth function $f: TQ\to {\mathbb R}$ is the map, denoted by
$d_{\vert}f$, which associates, 
to each $v\in TQ$, the differential at $v$ of the restriction of $f$ to the fibre $T_{\tau_Q(v)}Q$, 
where $\tau_Q:TQ\to Q$ is the canonical projection. 
We see that $d_{\vert}f(v)$ is an element of the dual of that fibre, which is the fibre 
over $\tau_Q(v)$ of the cotangent bundle $T^*Q$. Therefore $d_{\vert} f:TQ\to T^* Q$ is a
bundles homomorphism over the identity of $Q$ (but not a \emph{vector} bundles homomorphism 
since its restriction to a fibre
is not linear, except when the restriction of $f$ to that fibre is a quadratic form).
The \emph{Legendre map} ${\mathcal L}:TQ\to T^*Q$ associated to the Lagrangian $L$ is the vertical
differential $d_{vert} L$.
\par\smallskip

The partial differential of the function $\overline L:Q\times{\frak g}\to{\mathbb R}$ with respect to its second variable, 
which plays an important part in the Euler-Poincar\'e equation, can be expressed in terms of the momentum and Legendre maps. Indeed, according to its very definition,
 $$\d_{\frak g}\overline L=p_{{\frak g}^*}\circ\varphi^T\circ{\mathcal L}\circ\varphi\,.$$
Since $J=p_{{\frak g}^*}\circ\varphi^T$, we have
 $$\d_{\frak g}\overline L=J\circ{\mathcal L}\circ\varphi\,.$$

\subsection{Another Form of the Euler-Poincar\'e Equation.}
The Euler-Poincar\'e equation can therefore be written under the form
 $$\left(\frac{\d }{\d t}-\ad^*_{V(t)}\right)\Bigl(J\circ{\mathcal L}\circ\varphi\big(\gamma(t),V(t))\Bigr)
   =J\circ \d_Q\overline L\big(\gamma(t),V(t)\bigr)\,.\eqno\hbox{\bf (E-P2)}$$
Of course, together with that equation, we must consider
the compatibility condition
 $$\frac{\d\gamma(t)}{\d t}=\varphi\bigl(\gamma(t),V(t))\,,\eqno\hbox{\bf (CC)}$$

\begin{rmk} {\rm The compatibility condition $\hbox{\bf (CC)}$ is solved with respect to
$\displaystyle\frac{\d\gamma(t)}{\d t}$. But just like the usual Euler-Lagrange equation,  
the Euler-Poincar\'e equation $\hbox{\bf (E-P2)}$ is not solved with
respect to $\displaystyle\frac{\d V(t)}{\d t}$. Moreover, when written in local coordinates,
it appears as a system of $r$ differential equations for the $r$ 
components of the map $V:[t_0,t_1]\to{\frak g}$, of which at most $n$ can be independent: 
we have seen indeed that the lift to $Q\times{\frak g}$ of a given parametrized curve 
$\gamma:[t_0,t_1]\to Q$ is determined only up to adddition of an arbitrary map $W:[t_0,t_1]\to{\frak g}$
whose value $W(t)$ at each $t$ belongs to the isotropy Lie subalgebra ${\frak g}_{\gamma(t)}$.
Therefore when $r>n$, equations $\hbox{\bf (E-P2)}$ and $\hbox{\bf (CC)}$ form an \emph{under-determined} 
system of differential equations for the pair of unknown maps $t\mapsto\bigl(\gamma(t),V(t)\bigr)$, 
partially in implicit form.
}
\end{rmk}

\begin{rmk} {\rm Let us assume now that the Lagrangian $L$ depends on time, \emph{i.e.}, is a smooth function $L$ defined on the product $\RR\times TQ$, the variable $t\in \RR$ being the time. The other assumptions being unchanged, it is easy to see that the Euler-Poncar\'e equation remains valid, its proof being essentially the same as that given above. Of course,
$\overline L$ is now defined as
 $$\overline L(t,x,X)=L\bigl(t,\varphi(x,X)\bigr)\,,\quad 
  t\in\RR\,,\ x\in Q\,,\ X\in{\frak g}\,,
 $$
the Legendre map ${\mathcal L}$ is now defined on $\RR\times TQ$ and takes its values in 
$T^*Q$, and the Euler-Poincar\'e equation's expression becomes}
 $$\left(\frac{\d }{\d t}-\ad^*_{V(t)}\right)\Bigl(J\circ{\mathcal L}
    \bigl(t,\varphi\big(\gamma(t),V(t)\bigr)\Bigr)
   =J\circ \d_Q\overline L\big(t,\gamma(t),V(t)\bigr)\,.\eqno\hbox{\bf (E-P3)}
 $$
\end{rmk}

\subsection{Example: the spherical pendulum}\label{sphericalpendulum}
Let us consider an heavy material point of mass $m$ constrained, by an ideal constraint, on the surface of a 
sphere $Q$ of centre $O$ and radius $R$, embedded in the physical space $E$. As in Subsection~\ref{rigidbody}, once units of length and time are chosen we may consider the physical space $E$ as an oriented Euclidean 
three-dimensional vector space, with $O$ as origin. The Lie group $\SO(E)$ acts on $E$ on the left by an action which maps the sphere $Q$ onto itself, therefore acts on $Q$. Using the Euclidean scalar product of 
vectors in $E$, we can identify
a point of $Q$ with a vector $\vect x\in E$ of length $R$, and a vector tangent at $\vect x$ to the sphere $Q$
with a pair $(\vect x, \vect v)$ of vectors in $E$ satisfying
 $$\vect x.\vect x=R^2\,,\quad \vect x.\vect v=0\,.$$ 
The choice of an orientation of $E$ allows us to identify the Lie algebra ${\frak g}=\frak{so}(E)$ 
with the vector space $E$ itself, the bracket of elements in
${\frak g}$ being expressed by the vector product of the corresponding vectors in $E$. The map $\varphi:Q\times{\frak g}\to TQ$ can be expressed as
 $$\varphi(\vect x, \vect\Omega)=\vect\Omega_Q(\vect x)=\vect\Omega \times\vect x\,.$$  
The Euclidean scalar product allows us to identify the tangent bundle $T^*Q$ with the cotangent bundle $T^*Q$
and the Lie algebra $\frak{g}$ with its dual $\frak{g}^*$. We may even consider a pair
$(\vect x,\vect \eta)$ of vectors in $E$, the first one $\vect x$ being of length $R$, 
as the element in $T^*_{\vect x}Q$ which evaluated on the tangent vector $(\vect x,\vect v)$, takes the value 
$\vect\eta.\vect v$. The scalar product $\vect x.\vect\eta$ is not assumed to be zero, but of course the element of $T^*_{\vect x}Q$ defined by $(\vect x,\vect \eta)$ only depends on $\vect x$ and of 
$\displaystyle \vect\eta-\frac{\vect x.\vect\eta}{R^2}\vect x$. 
\par\smallskip

The Lagrangian of the system is
 $$L(\vect x,\vect v)=\frac{m\Vert \vect v\Vert^2}{2}+m\vect g.\vect x\,,$$
where $\vect g$ is the acceleration of gravity (considered as a vertical vector directed downwards). The function
$\overline L=L\circ\varphi:Q\times\frak{g}\to{\mathbb R}$ is therefore
 $$\overline L(\vect x,\vect\Omega)=\frac{mR^2}{2}\left(\Vert\vect\Omega\Vert^2
    -\frac{(\vect\Omega.\vect x)^2}{R^2}\right)\,.$$
By calculating the partial differentials of $\overline L$, we easily see that the Euler-Poincar\'e equation 
$\hbox{\bf (E-P1)}$  becomes 
 $$\frac{d}{dt}\bigl(mR^2\vect\Omega-m(\vect x.\vect\Omega)\vect x\bigr)=m\vect x\times \vect g\,.
 $$
This equation can easily be obtained by much more elementary methods: it expresses the fact that the time derivative of the momentum at the origin is equal to the momentum at that point of the gravity force (since the momentum at the origin of the constraint force exerted on the material point by the surface of the sphere vanishes).
\section{The Euler-Poincar\'e Equation and Reduction}\label{EPR}

Following the remark made by Poincar\'e at the end of his note, let us now assume that the map 
$\overline L:Q\times{\frak g}\to {\mathbb R}$ only depends on its second variable $X\in {\frak g}$. 
In other words, $\overline L$ is assumed to be a function defined on ${\frak g}$. 
Its partial differential with respect to its first variable  $\d_Q\overline L$ 
therefore vanishes, and its partial differential with respect to its second variable
$\d_{\frak g}\overline L$ is its usual differential $\d \overline L$. The Euler-Poincar\'e equation becomes
 $$\left(\frac{\d }{\d t}-\ad^*_{V(t)}\right)\Bigl(\d \overline L\bigl(V(t)\bigr)\Bigr)=0\,.\eqno\hbox{\bf (E-P3)}
 $$
This form of the Euler-Poincar\'e equation, called the \emph{basic Euler-Poincar\'e equation} in \cite{holm}, only contains the unknown map $t\mapsto V(t)$, but is not solved with respect to $\displaystyle\frac{\d V(t)}{\d t}$: it is an \emph{implicit differential equation} for the unknown map $V$. Moreover, we know that when
$\dim{\frak g}>\dim Q$, it is \emph{underdetermined}, since we have seen that $V$ is determined only up to addition of a map whose value, for each $t$, belongs to the isotropy algebra of $\gamma(t)$. 
\par\smallskip

Once a solution $V$ of $\hbox{\bf (E-P3)}$ is found, it can be inserted in the compatibility condition 
$\hbox{\bf (CC)}$ which becomes a differential equation
(in explicit form) for the still unknown map $t\mapsto \gamma(t)$. Solving that equation is sometimes called 
\emph{reconstruction} by modern authors. 
\par\smallskip

We see that when $\overline L$ is a function defined on $\frak g$, the determination of motions of our mechanical system
can be simplified by the use of the Euler-Poincar\'e equation.  Usually, it involves the resolution of an implicit differential equation (the Euler-Lagrange equation) on the $2n$-dimensional manifold $TQ$. Now it can be made in two steps:
first by solving the Euler-Poincar\'e equation $\hbox{\bf (E-P3)}$, which is an implicit differental equation for the unknown
$V$ on the $r$-dimensional vector space $\frak g$; and then by solving the compatibility condition $\hbox{\bf (CC)}$, 
which is an explicit differential equation for the unknown $\gamma$ on the $n$-dimensional manifold $Q$.
This procedure is called \emph{Lagrangian reduction} in \cite{cendra,holm,holmmarsden,ratiu}. In our opinion this name
is inappropriate: we will see in Section~\ref{Ham} that this procedure can be used in the Hamiltonian formalism
as well as in the Lagrangian formalism. 
\par\smallskip

\begin{rmk}{\rm
The assumption that $\overline L=L\circ\varphi$ only depends on its second variable
\emph{does not mean} that the Lagrangian $L$ is invariant with respect to the action on $TQ$ of the Lie algebra
$\frak g$ lifted from its action on $Q$. One can prove indeed that the Lie derivative of 
$L$ with respect to the fundamental vector field on $TQ$ associated to a given element $X\in{\frak g}$ 
\emph{generally does not vanish}. Even when the Lie algebra action of ${\frak g}$ 
on the manifold $Q$ comes from the action of a Lie group $G$ on that manifold, the Lagrangian $L$ 
\emph{generally is not constant} on each orbit of the action of $G$ on $TQ$ lifted from its action on $Q$.
The true meaning of the assumption that $\overline L$ is a function defined on 
$\frak g$ is given by the following Lemma.
}
\end{rmk}

\begin{lemma}\label{stat1}
The map $\overline L=L\circ\varphi:Q\times{\frak g}\to {\mathbb R}$ is a 
function defined on $\frak g$ only if and only if, for each $X\in{\frak g}$, the Lagrangian $L$  is constant 
on the image (sometimes improperly called the graph) $X_Q(Q)$ of the fundamental vector field $X_Q$. 
Moreover, when this condition is satisfied, the momentum map $J$ is constant on the 
image ${\mathcal L}\circ X_Q(Q)$ of the map $\mathcal L\circ X_Q$. 
\end{lemma}

\begin{proof}
For any given $x\in Q$ and $X\in{\frak g}$
 $$\overline L(x,X)=L\circ\varphi(x,X)=L\bigl(X_Q(x)\bigr)\,,$$
which proves that $\overline L$ is a function defined on $\frak g$ only if and only if for each
$X\in{\frak g}$ $L$ is constant on the submanifold $X_Q(Q)$ of $TQ$. Moreover when this condition is satisfied,
for all $x\in Q$, $X$ and $Y\in{\frak g}$,
  $$\bigl\langle J\circ{\mathcal L}\circ X_Q(x),Y\bigr\rangle=\bigl\langle \d_{\frak g}\overline L(x,X),Y\bigr\rangle
     =\bigl\langle \d \overline L(X),Y\bigr\rangle\,,$$
which proves that for each $Y\in {\frak g }$, $\langle J,Y\rangle$ is constant on the subset ${\mathcal L}\circ X_Q(Q)$ of $T^*Q$,
which means that $J$ itself is constant on that subset.
\end{proof}

\begin{rmk}{\rm
When the dimension of the Lie algebra $\mathcal G$ is strictly larger than the dimension of the configuration space $Q$, very strong restrictions limit the applicability of Lagrangian reduction. Let us consider for example
a mechanical system whose configuration space $Q$ is a sphere of radius $R$ embedded in the physical space $E$, as in Subsection \ref{sphericalpendulum}. The only Lagrangians $L$ on $TQ$ which are such that $\overline L=L\circ\varphi$ is a function defined on $\frak{g}$ are constants. We have seen indeed that the Lie algebra $\frak g$ can be identified with the Euclidean vector space $E$. Let $\vect\Omega_1$ and $\vect\Omega_2$ be two distinct elements of $\frak g$, and let 
$\displaystyle \vect x=\frac{R(\vect\Omega_2-\vect\Omega_1)}{\Vert\vect\Omega_2-\vect\Omega_1\Vert}$.
The vector $\vect x$, which can be considered as a point on the sphere $Q$, is such that 
$\vect\Omega_2-\vect\Omega_1$ lies in its isotropy subalgebra, since that vector is normal to the plane tangent at $\vect x$ to the sphere $Q$. Therefore if we assume that 
$\overline L(\vect x,\vect\Omega)$ only depends on $\vect\Omega$, not on $\vect x$, we must have
$\overline L(\vect\Omega_2)=\overline L(\vect\Omega_1)$, and we see that the Lagrangian $L$ must be a constant.
} 
\end{rmk}

\section{The Euler-Poincar\'e Equation in Hamiltonian Formalism}\label{Ham}

\subsection{The Hamiltonian.}
No special assumption was made until now about the regularity of the Lagrangian $L$. 
Now we assume that $L$ is \emph{hyperregular}, 
which means that the Legendre map $\mathcal L$ is a diffeomorphism of $TQ$ onto the phase space $T^*Q$.
We can then define a smooth function $H:T^*Q\to {\mathbb R}$, called the
\emph{Hamiltonian}, given by
  $$H(\xi)=\bigl\langle\xi,{\mathcal L}^{-1}(\xi)\bigr\rangle - L\bigl({\mathcal L}^{-1}(\xi)\bigr)\,,
    \quad \xi\in T^*Q\,.$$

\subsection{Lagrangian, Hamiltonian and Euler-Poincar\'e Formalisms.}
The \emph{La\-gran\-gian formalism} is the mathematical description of motions of our me\-cha\-ni\-cal system as smooth 
parametrized curves $\gamma:[t_0,t_1]\to Q$ at which the action functional 
 $$I_L(\gamma))=\int_{t_0}^{t_1}L\left(\frac{\d\gamma(t)}{\d t}\right)\,\d t$$
is stationary with respect to variations of $\gamma$ with fixed endpoints. As we have seen in Section 2, Poincar\'e has proven that
the Lagrangian formalism is equivalent to the \emph{Euler-Poincar\'e formalism}, that means the mathematical 
description of motions as smooth parametrized curves $(\gamma,V):[t_0,t_1]\to Q\times{\frak g}$ which satisfy
the Euler-Poincar\'e equation $\hbox{\bf (E-P1)}$ and the compatibility condition $\hbox{\bf (CC)}$. 
\par\smallskip

The \emph{Hamiltonian formalism} is the mathematical description of motions of our mechanical system as smooth parametrized curves
$\zeta:[t_0,t_1]\to T^*Q$ which satisfy the \emph{Hamilton equation}, \emph{i.e.}, the differential equation associated to the Hamiltonian vector field  ${\mathcal X}_H$,
 $$\frac{\d \zeta(t)}{\d t}={\mathcal X}_H\bigl(\zeta(t)\bigr)\,.\eqno\hbox{\bf (H)}$$
Since $L$ is assumed to be hyperregular, the Lagrangian formalism and the Hamiltonian formalisms are equivalent. Therefore, the Euler-Poincar\'e formalism too is equivalent to the Hamiltonian formalism.
\par\smallskip

When $\dim{\frak g}=\dim Q$, the equivalence between the Euler-Poincar\'e and the Hamiltonian formalisms is easily understood,
since the vector bundle homomorphism $\varphi:Q\times{\frak g}\to TQ$ is an isomorphism; its transpose
$\varphi^T=(\pi_Q,J):T^*Q\to Q\times{\frak g}^*$ too is an isomorphism. The Euler-Poincar\'e equation can be written
on $Q\times{\frak g}^*$, and appears then as the image by the isomorphism $(\pi_Q,J)$ of the Hamilton equation on $T^*Q$.
\par\smallskip

Things are more complicated when $\dim{\frak g}>\dim Q$. 
The vector bundles homomorphism $\varphi:Q\times{\frak g}\to TQ$ 
is surjective but no more injective: its kernel is the vector sub-bundle 
of $Q\times{\frak g}$ whose fibre over each point $x\in Q$ is the
isotropy sub-algebra ${\frak g}_x$ of that point. Therefore its transpose
$\varphi^T=(\pi_Q,J):T^*Q\to Q\times{\frak g}^*$ is
an injective, but no more surjective  vector bundles homomorphism. Its image
is the vector sub-bundle $\Upsilon$ of $Q\times{\frak g}^*$ whose fibre over each point $x\in Q$ 
is the annihilator $({\frak g}_x)^0$ of ${\frak g}_x$. The dimension of $\Upsilon$ is $2n$. 
We can choose a vector sub-bundle of $Q\times{\frak g}$ whose fibre over each point $x\in Q$ 
is a vector subspace of $\frak g$ complementary to the isotropy
Lie sub-algebra ${\frak g}_x$, for example by choosing a symmetric 
positive definite bilinear form on $\frak g$, and
taking for fibre over each $x\in Q$ the orthogonal of 
${\frak g}_x$ with respect to that form. The total space 
$\Gamma$  of that sub-bundle is a $2n$-dimensional manifold. The map
 $$(\pi_Q,J)\circ{\mathcal L}\circ\varphi:Q\times{\frak g}\to Q\times{\frak g}^*$$
restricted to $\Gamma$ is now a diffeomorphism of $\Gamma$ onto its image $\Upsilon$.
The Euler-Poincar\'e equation can be written on $\Upsilon$, and appears then as the image by
$(\pi_Q,J)$ (considered as taking its values in $\Upsilon$) of the Hamilton equation on $T^*Q$.

\section{Euler-Poincar\'e Reduction in the Hamiltonian Formalism.}\label{EPRH}
The Lagrangian $L$ is still assumed to be hyperregular, and in addition such that $\overline L=L\circ\varphi$
is a function of its second variable only, that means a function defined on $\frak g$. As in Section~\ref{EPR}, we therefore
have $\d_Q\overline L=0$ and $\d_{\frak g}\overline L=\d \overline L$. We know by Lemma~\ref{stat1} that for each $X\in{\frak g}$
the Lagrangian $L$ is constant on the submanifold $X_Q(Q)$ of $TQ$, and the momentum map $J$ is constant on the
submanifold ${\mathcal L}\circ X_Q(Q)$ of $T^*Q$. The next Lemma shows that
the Hamiltonian has a similar invariance property.

\begin{lemma}\label{stat2}
When the Lagrangian $L$ is assumed to be hyperregular and such that $\overline L$ is a function 
defined on $\frak g$
only, for each $X\in{\frak g}$ the Hamiltonian $H$ is constant on the submanifold ${\mathcal L}\circ X_Q(Q)$ of $T^*Q$. 
\end{lemma}

\begin{proof}
For a given $X\in{\frak g}$ and alla $x\in Q$, we have
 $$H\circ{\mathcal L}\circ X_Q(x)=\bigl\langle{\mathcal L}\circ X_Q(y), X_Q(y)\bigr\rangle-L\circ X_Q(y)
=\bigl\langle \d \overline L(X),X\bigr\rangle -\overline L(X)\,,$$
which proves that $H$ is constant on the submanifold ${\mathcal L}\circ X_Q(Q)$ of $T^*Q$. 
\end{proof}

\subsection{The Euler-Poincar\'e Equation on ${\frak g}^*$.} 
With the assumptions made in this section, the Euler-Poincar\'e and the Hamiltonian formalisms are equivalent; therefore
a smooth parametrized curve $(\gamma,V):[t_0,t_1]\to Q\times{\frak g}$ which satisfies the compatibility condition
$\hbox{\bf (CC)}$ also satisfies the Euler-Poincar\'e equation $\hbox{\bf (E-P1)}$ if and only if the parametrized curve
$\zeta={\mathcal L}\circ\varphi\circ(\gamma,V):[t_0,t_1]\to T^*Q$ satisfies the Hamilton equation $(\hbox{\bf H})$.
The Euler-Poincar\'e equation becomes  
$$\left(\frac{\d }{\d t}-\ad^*_{V(t)}\right)\Bigl(J\circ\zeta(t)\Bigr)
 =0\,.\eqno\hbox{\bf (E-P4)}$$
This equation shows that when $(\gamma,V)$ satisfies $\hbox{\bf (CC)}$ and $\hbox{\bf (E-P1)}$, the parametrized curve
$\xi=J\circ\zeta=J\circ{\mathcal L}\circ\varphi\circ(\gamma,V)$ takes its value in a coadjoint orbit of ${\frak g}^*$. 
One may wish to consider it as a differential equation in ${\frak g}^*$ for the unknown parametrized
curve $\xi:[t_0,t_1]\to{\frak g}^*$. But there are at least two difficulties.
\par
\begin{enumerate}
\item{} 
The term $\ad^*_{V(t)}$ depends of $V(t)$, which is an element in ${\frak g}$ whose dependence
on $\xi(t)=J\circ\zeta(t)$ is complicated. Of course, we can write 
 $$\gamma(t)=\pi_Q\circ\zeta(t)\quad\hbox{and}
   \quad \bigl(\gamma(t),V(t)\bigr)\in \varphi^{-1}\circ{\mathcal L}^{-1}\bigl(\zeta(t)\bigr)\,,$$
which proves that when $\zeta(t)$ is known, $V(t)$ is determined up to an element in the isotropy Lie algebra
${\frak g}_{\pi_Q\bigl(\zeta(t)\bigr)}$. Still, $V(t)$ is not fully determined by $\xi(t)$.

\item{}
When $r=\dim{\frak g}$ is strictly larger than $n=\dim Q$, for each $x\in Q$ the map $J:T^*Q\to{\frak g}^*$
restricted to $T^*_xQ$ is injective, but not surjective: its image is the annihilator $({\frak g}_x)^0$ 
of the isotropy Lie algebra ${\frak g}_x$. Therefore Equation $\hbox{\bf (E-P4)}$ 
may not be well defined on the whole vector space ${\frak g}^*$. 
\end{enumerate} 
\par\smallskip

However, let us recall that ${\frak g}^*$ has a natural Poisson structure, called the 
\emph{Kirillov-Kostant-Souriau structure}, which allows to associate to any smooth function 
$h:{\frak g}^*\to{\mathbb R}$ its \emph{Hamiltonian vector field} ${\mathcal X}_h$, whose expression is
$${\mathcal X}_h(\xi)=-\ad^*_{dh(\xi)}\xi\,,\quad \xi\in{\frak g}^*\,.
 $$
Moreover, the momentum map $J:T^*Q\to {\frak g}^*$ is a Poisson map when
$T^*Q$ is endowed with the Poisson structure associated to its symplectic structure, and ${\frak g}^*$ with its
Kirillov-Kostant-Souriau Poisson structure~(\cite{lima}, chapter IV proposition 5.2). 
Therefore, if there exists a smooth function $h:{\frak g}^*\to {\mathbb R}$ such that $H=h\circ J$, 
the parametrized curve $\xi=J\circ\zeta:[t_0,t_1]\to{\frak g}^*$
satisfies the Hamilton differential equation on ${\frak g}^*$
 $$\frac{\d \xi(t)}{\d t}= -\ad^*_{dh\bigl(\xi(t)\bigr)}\bigl(\xi(t)\bigr)\,. \eqno\hbox{\bf(E-P5)}$$
It is Equation~$\hbox{\bf(E-P4)}$ with $V(t)=-dh\bigl(\xi(t)\bigr)$. We see therefore that 
Equation~$\hbox{\bf(E-P4)}$ becomes a well-defined differential equation on ${\frak g}^*$, containing no
unknown other than the parametrized curve $\xi=J\circ\zeta:[t_0,t_1]\to{\frak g}^*$, 
\emph{if and only if the Hamiltonian $H:T^*Q\to {\mathbb R}$ is a map obtained by composition of the momentum map 
$J:T^*Q\to{\frak g}^*$} with a smooth map $h:{\frak g}^*\to {\mathbb R}$.   
While the better known Marsden-Weinstein reduction procedure~\cite{marsdenweinstein} is used when the momentum map is a first integral of the Hamilton equation $\hbox{\bf (H)}$, in the Euler-Poincar\'e reduction procedure the momentum map $J$ need not be a first integral, but a different invariance property is needed: the Hamiltonian must be a map obtained by composition of the momentum map with a smooth map $h:{\frak g}^*\to {\mathbb R}$.
These two different reduction procedures are well known for Hamiltonian systems~(see for example \cite{lima} chapter IV section 6.11, or \cite{marle} last remark in Section~2). The invariance properties used by these two different reduction procedure
are related by the fact that, for each $\zeta\in T^*Q$, each of the the two vector subspaces $\ker T_\zeta J$ 
and $T_\zeta{\mathcal O}_\zeta$ of the tangent space $T_\zeta(T^*Q)$ is the symplectic orthogonal of the other. 
We have denoted by $T_\zeta{\mathcal O}_\zeta$ the tangent space at $\zeta$ to the orbit of $\zeta$ under 
the action of $\frak g$ on $T^*Q$, that means the vector subspace of $T_\zeta(T^*Q)$ made by the values at 
$\zeta$ of the fundamental vector fields $X_{T^*Q }$, for all $X\in{\frak g}$.

\section{Systems whose Configuration Space is a Lie Group}\label{confspaceliegroup}
In this section we assume that
$Q$ is a connected Lie group $G$ whose Lie algebra
(identified with the tangent space to $G$ at the neutral element) is
$\frak g$, and that the Lie algebra action $\psi$ is the map which associates to each 
$X\in {\frak g}$ the \emph{right invariant} vector field $X_G^R$ on $Q\equiv G$
whose value at the neutral element is $X$. Therefore 
$n=\dim Q=\dim{\frak g}=r$.
\par\smallskip
  
Let us first recall some well known results about the actions of a Lie group on itself and their lifts to the tangent and cotangent bundles.

\subsection{Actions of a Lie Group on itself on the Right and on the Left.}
For each $g\in G$, we denote by $R_g:G\to G$ and $L_g:G\to G$ the left and right translations
 $$R_g(x)=xg\,,\quad L_g(x)=gx\,,\quad x\in G\,,$$
and by $TR_g:TG\to TG$ and $TL_g:TG\to TG$ their prolongations to vectors.
\par\smallskip

Observe that $X_G^R$ is the fundamental vector field
associated to $X$ for the action of $G$ on itself by translations \emph{on the left} 
 $$\Phi^L:G\times G\to G\,,\quad \Phi^L(g,x)=gx\,,$$
not on the right, since
 $$X_G^R(x)=\frac{\d \bigl(\exp(tX)x\bigr)}{\d t}\Bigm|_{t=0}=TR_x(X)\,,\quad X\in{\frak g}\,,\quad x\in G\,.$$
 The vector bundles morphism $\varphi:G\times{\frak g}\to TG$ is now an isomorphism, given by 
 $$\varphi(x,X)=X^R_G(x)=TR_x(X)\,,\quad x\in G\,,\quad X\in {\frak g}\,.$$
To define the lift to the tangent bundle $TG$ of the action $\Phi^L$, we take for 
each $g\in G$ the prolongation to vectors of the diffeomorphism $L_g:G\to G$, $x\mapsto L_g(x)=gx$.
The obtained action of $G$ on $TG$, denoted by $\overline\Phi^L:G\times TG\to TG$, is given by
 $$\overline\Phi^L(g,v)=TL_g(v)\,,\quad g\in G\,,\quad v\in TG\,.$$
The lift to the cotangent bundle $T^*G$ of the action $\Phi^L$, denoted by $\widehat \Phi^L$, is the contragredient of
$\overline\Phi^L$, with a change of sign (to obtain an action on the left):
 $$\widehat\Phi^L(g,\zeta)=(TL_{g^{-1}})^t(\zeta)\,,\quad g\in G\,,\quad \zeta\in T^*G\,.$$ 
We have denoted by $(TL_{g^{-1}})^t:T^*G\to T^*G$ the transpose of the linear vector bundles isomorphism 
$TL_g:TG\to TG$.
\par\smallskip  

The action $\widehat \Phi^L:G\times T^*G\to T^*G$ is Hamiltonian (see for example \cite{lima}, chapter IV, theorem 4.6), and has as an $\Ad^*$-invariant momentum map
 $$J^L:T^*G\to{\frak g}^*\,,\quad J^L(\zeta)=\bigl(TR_{\pi_G(\zeta)}\bigr)^t(\zeta)\,,\quad \zeta\in T^*G\,.$$

Let us now consider the action of the Lie group $G$ on itself by translations \emph{on the right}
 $$\Phi^R:G\times G\to G\,,\quad \Phi^R(x,g)=xg\,.$$
For this action, the fundamental vector field associated to each $X\in{\frak g}$ is the \emph{left invariant}
vector field $X_G^L$ on $G$ whose value at the neutral element is $X$. The lift to $TG$ and to $T^*G$ of the action $\Phi^R$,
denoted respectively $\overline\Phi^R:TG\times G\to TG$ and $\widehat\Phi^R:T^*G\times G\to T^*G$, are given by
the formulae, in which $v\in TG$, $\zeta\in T^*G$, $g\in G$,
 $$\overline\Phi^R(v,g)=TR_g(v)\,,\quad \widehat\Phi^R(\zeta,g)=(TR_{g^{-1}})^t(\zeta)\,.$$
The action $\widehat \Phi^R$ is Hamiltonian and admits as an $\Ad^*$-invarant momentum map
 $$J^R(\zeta)=\bigl(TL_{\pi_G(\zeta)}\bigr)^t(\zeta)\,,\quad \zeta\in T^*G\,.$$

\begin{proposition}\label{stat3}
A smooth Hamiltonian $H:T^*G\to{\mathbb R}$ can be written as $H=h\circ J^L$, 
where $h:{\frak g}^*\to{\mathbb R}$ is a smooth map, if and only if
$H$ remains invariant by the action $\widehat \Phi^R$. When the Hamiltonian $H$ comes from a hyperregular smooth
Lagrangian $L:TG\to {\mathbb R}$, $H$ can be written as $H=h\circ J^L$ if and only if 
the Lagrangian $L$ is such that the function
 $$\overline L=L\circ\varphi:G\times{\frak g}^*\to{\mathbb R}\,,\quad
   (x,X)\mapsto \overline L(x,X)=L\bigl(TR_x(X)\bigr)$$ 
(with $x\in G$, $X\in{\frak g}$) is a function defined on ${\frak g}^*$, or if and only if $L$ remains invariant by the action $\overline\Phi^R$.
\end{proposition}

\begin{proof}
The above given expressions
of the actions $\widehat\Phi^L$ and $\widehat\Phi^R$ and of their momentum maps 
$J^L$ and $J^R$ prove that the level sets $(J^L)^{-1}(\xi)$ of the map $J^L$, for all $\xi\in{\frak g}^*$,
are the orbits of the action $\widehat\Phi^R$ (and the level sets of 
$J^R$ are the orbits of the action $\widehat\Phi^L$). Therefore they are $n$-dimensional smooth submanifolds of $T^*G$ diffeomorphic to $G$. Since $J^L:T^*G\to{\frak g}^*$ is a surjective submersion,
the Hamiltonian $H$ can be written as $H=h\circ J^L$, where $h:{\frak g}^*\to{\mathbb R}$ is a smooth map, if and only if
it takes a constant value on each level set of $J^L$, \emph{i.e.}, on each orbit of $\Phi^R$, in other words if and only if
$H$ remains invariant by the action $\widehat \Phi^R$. The relationship between a hyperregular Lagrangian $L$ and the corresponding Hamiltonian $H$ shows that the invariance of $H$ by the action $\widehat\Phi^R$ is equivalent to the invariance of $L$ by the action $\overline\Phi^R$.  
\end{proof}

\begin{rmk}{\rm
When the Hamiltonian can be written as $H=h\circ J^L$,  where $h:{\frak g}^*\to{\mathbb R}$ is a smooth function, 
the Euler-Poincar\'e reduction procedure allows the resolution of the Euler-Poincar\'e equation
$\hbox{\bf (E-P5)}$ as a first step to solve the Hamilton equation $\hbox{\bf (H)}$. On the other hand, 
Noether's theorem~(see for example \cite{lima}, Chapter IV, Theorem 2.6) 
asserts that $J^R$ is a first integral of $\hbox{\bf (H)}$, and allows the use of the Marsden-Weinstein 
reduction procedure. We see therefore that the assumptions under which the Euler-Poincar\'e and Marsden-Weinstein 
reduction procedures can be used are the same. The next proposition will allow us to prove that 
these two reduction procedures are equivalent.
}
\end{rmk}

\begin{proposition}\label{stat4}
For each $\zeta\in {\frak g}^*$, $(J^R)^{-1}\bigl(J^R(\zeta)\bigr)$ and $(J^L)^{-1}\bigl(J^L(\zeta)\bigr)$ are two 
smooth submanifolds of $T^*G$, diffeomorphic to $G$, which are the orbits through $\zeta$ of the actions
$\widehat\Phi^L:G\times T^*G\to T^*G$ and $\widehat \Phi^R:T^*G\times G\to T^*G$, respectively. Their intersection
is a smooth isotropic submanifold of the symplectic manifold $(T^*G,\omega_{T^*G})$, which can be written both as
 $$(J^R)^{-1}\bigl(J^R(\zeta)\bigr)\cap (J^L)^{-1}\bigl(J^L(\zeta)\bigr)
   =\bigl\{\,(TL_{g^{-1}})^t(\zeta)\,;\,g\in G_{J^L(\zeta)}\,\bigr\}\,,$$
and as
 $$(J^R)^{-1}\bigl(J^R(\zeta)\bigr)\cap (J^L)^{-1}\bigl(J^L(\zeta)\bigr)
   =\bigl\{\,(TR_{\gamma^{-1}})^t(\zeta)\,;\,\gamma\in G_{J^R(\zeta)}\,\bigr\}\,,$$
where $G_{J^L(\zeta)}$ and $G_{J^R(\zeta)}$ are the isotropy groups, respectively, 
of $J^L(\zeta)$ and of $J^R(\zeta)$ for the coadjoint action of $G$ on ${\frak g}^*$.

Moreover, the tangent space at $\zeta$ to $(J^R)^{-1}\bigl(J^R(\zeta)\bigr)\cap (J^L)^{-1}\bigl(J^L(\zeta)\bigr)$ 
is the kernel of the closed $2$-form induced by $\omega_{T^*G}$ on the submanifolds
$(J^R)^{-1}\bigl(J^R(\zeta)\bigr)$ and $(J^L)^{-1}\bigl(J^L(\zeta)\bigr)$.
\end{proposition}

\begin{proof}
We already know that the level sets of $J^R$ are the orbits of $\widehat\Phi^L$ and that the level sets of
$J^L$ are the level sets of $\widehat\Phi^R$. These actions being free, these level sets are 
smooth submanifolds diffeomorphic to $G$. 
\par\smallskip

The level subset $(J^L)^{-1}\bigl(J^L(\zeta)\bigr)$ is the set of elements $(TR_{\gamma^{-1}})^t(\zeta)$, 
for all $\gamma\in G$. Let us calculate
 $$J^R\bigl((TR_{\gamma^{-1}})^t(\zeta)\bigr)=(TL_{\pi_G(\zeta)\gamma})^t(TR_{\gamma^{-1}})^t(\zeta)
    =\Ad^*_{\gamma^{-1}}\circ J^R(\zeta)\,.$$
Therefore $(TR_{\gamma^{-1}})^t(\zeta)$ belongs to $(J^R)^{-1}\bigl(J^R(\zeta)\bigr)$ 
if and only if $\Ad^*_{\gamma^{-1}}\circ J^R(\zeta)=J^R(\zeta)$, \emph{i.e.}, if and only if $\gamma$ belongs 
to $G_{J^R(\zeta)}$, the isotropy subgroup of $J^R(\zeta)$ for the coadjoint action of $G$ on ${\frak g}^*$. 
We have proved that
 $$(J^R)^{-1}\bigl(J^R(\zeta)\bigr)\cap (J^L)^{-1}\bigl(J^L(\zeta)\bigr)=
   \bigl\{\,(TR_{\gamma^{-1}})^t(\zeta)\,;\,\gamma\in G_{J^R(\zeta)}\,\bigr\}\,.
 $$ 
A similar calculation shows that
 $$J^L(\zeta)=\Ad^*_{\pi_G(\zeta)}\circ J^R(\zeta)\,,$$
from which we deduce that 
 $$G_{J^L(\zeta)}=\pi_G(\zeta)G_{J^R(\zeta)}\bigl(\pi_G(\zeta)\bigr)^{-1}$$
and that
$$(J^R)^{-1}\bigl(J^R(\zeta)\bigr)\cap (J^L)^{-1}\bigl(J^L(\zeta)\bigr)=
   \bigl\{\,(TL_{g^{-1}})^t(\zeta)\,;\ g\in G_{J^L(\zeta)}\,\bigr\}\,.
 $$ 
Finally, let us recall that when a Lie group $G$ acts on a 
symplectic manifold $(M,\omega)$ by a Hamiltonian action which admits a map $J$ as momentum map,
for each point $x\in M$ each of the two vector subspaces of $T_xM$: (i) the tangent space at $x$ 
to the orbit of this point, and (ii) $\ker T_xJ$, is the symplectic orthogonal of the other.
Therefore, for each $\zeta\in T^*M$, each of the tangent spaces at $\zeta$ to the submanifolds
$(J^R)^{-1}\bigl(J^R(\zeta)\bigr)$ and $(J^L)^{-1}\bigl(J^L(\zeta)\bigr)$ 
is the symplectic orthogonal of the other, and their intersection is the kernel at $\zeta$ of the closed 
$2$-forms induced by the canonical symplectic $2$form $\omega_{T^*G}$ on 
the submanifolds $(J^R)^{-1}\bigl(J^R(\zeta)\bigr)$ and $(J^L)^{-1}\bigl(J^L(\zeta)\bigr)$.
\end{proof}

\begin{corol}\label{stat5}
For each $\xi\in{\frak g}^*$, the submanifold $(J^R)^{-1}(\xi)$ is regularly foliated
by its intersections with the submanifolds $(J^L)^{-1}(\eta)$, where $\eta$ runs over ${\frak g}^*$.
The leaves of this foliation are the orbits of the action on $(J^R)^{-1}(\xi)$ of $\widehat\Phi^L$ restricted 
to the isotropy subgroup $G_\xi$ of $G$ (for the coadjoint action of $G$ on ${\frak g}^*)$.
These leaves are the maximal isotropic submanifolds of the symplectic manifold $(T^*G,\omega_{T^*G})$ contained in 
$(J^R)^{-1}(\xi)$. The set of leaves of that foliation is a
smooth manifold $M_\xi$, the canonical projection $p_\xi:(J^R)^{-1}(\xi)\to M_\xi$ is a smooth map, and
$M_\xi$ is  endowed with a symplectic form $\omega_\xi$ whose inverse image $p_\xi^*\omega_\xi$
is the closed $2$-form induced by $\omega_{T^*G}$ on $(J^R)^{-1}(\xi)$.

Moreover, the restriction to $(J^R)^{-1}(\xi)$ of the map $J^L:T^*G\to{\frak g}^*$ induces, by quotient,
a smooth map $J^L_\xi:M_\xi\to{\frak g}^*$, whose image is the coadjoint orbit of $\xi$. Considered as a map
defined on $M_\xi$ with values in the coadjoint orbit of $\xi$, the map $J^L_\xi$ is a symplectic diffeomorphism.
\end{corol}

\begin{proof}
For each $\zeta\in (J^R)^{-1}(\xi)$, we know by Proposition~\ref{stat4} that the kernel at $\zeta$ of the closed 
$2$-form induced on $(J^R)^{-1}(\xi)$ by $\omega_{T^*G}$ is the tangent space at $\zeta$ to the orbit through 
that point of the action $\widehat\Phi^L$ restricted to $G_\xi$. Its dimension, equal to
$\dim G_{\xi}$, does not depend on $\zeta$. Therefore, the rank of the closed $2$-form induced  
by $\omega_{T^*G}$ on $(J^R)^{-1}(\xi)$ is constant, and its kernel is an integrable vector sub-bundle of
$T\bigl((J^R)^{-1}(\xi)\bigr)$. The orbits of the action on $(J^R)^{-1}(\xi)$ of $\widehat\Phi^L$ restricted 
to $G_\xi$ are the leaves of the foliation $\mathcal F$ determined by this integrable vector sub-bundle. Therefore they are the
the maximal connected isotropic submanifolds contained in $(J^R)^{-1}(\xi)$. 
Let $M_\xi$ be the set of leaves of $\mathcal F$. For each $\zeta\in (J^R)^{-1}(\xi)$, $\ker T_\zeta J^L$ is the
tangent space at $\zeta$ to the orbit of $\widehat\Phi^R$ throught that point; its interserction with 
$T_\zeta\bigl((J^R)^{-1}(\xi)\bigr)$ is the tangent space at $\zeta$ to the leaf of $\mathcal F$ through
that point. Therefore, the map $J^L$ restricted to $(J^R)^{-1}(\xi)$ induces by quotient a map $J^L_\xi:M_\xi\to{\frak g}^*$,
which is injective and whose image is $\Ad^*_{G}(\xi)$, the coadjoint orbit of $\xi$. 
Considered as defined on $M_\xi$ with values in $\Ad^*_{G}(\xi)$, the map $J^L_\xi$ becomes bijective 
and can be used to transfer on $M_\xi$ the smooth manifold structure of $\Ad^*_{G}(\xi)$. We know that the
coadjoint orbits are the symplectic leaves of ${\frak g}^*$ endowed with its Kirillov-Kostant-Souriau
Poisson structure and that $J^L:T^*G\to{\frak g}^*$ is a Poisson map (see, for example, \cite{lima}, Chapter IV, Theorem 4.8, Remarks 4.9 and Proposition 5.2). Therefore the pull-back of the canonical symplectic form on
$\Ad^*_G(\xi)$ is a symplectic form on $M_\xi$ whose pull-back by the canonical projection
$p_\xi:(J^R)^{-1}(\xi)\to M_\xi$ is the $2$-form induced by $\omega_{T^*G}$ on ${J^R}^{-1}(\xi)$.
\end{proof}

\subsection{The Marsden-Weinstein and Euler-Poincar\'e Reduction Procedures.} 
In the Marsden-Weinstein reduction procedure~\cite{marsdenweinstein}, when the Hamiltonian $H$ 
remains invariant under the action $\widehat\Phi^R$, one considers the subset $(J^R)^{-1}(\xi)$ of $T^*G$ on which the
momentum map $J^R$ takes a given value $\xi\in{\frak g}^*$. In the present case, it is a submanifold,
image of the left-invariant $1$-form on $G$ whose value at the neutral element is $\xi$. Then one looks at the
subgroup made by elements $g\in G$ such that $\widehat\Phi^R_g$ maps this submanifold onto itself. 
We have seen (Corollary~\ref{stat5}) that it is $G_\xi$, and that the set of orbits of its action on
on $(J^R)^{-1}(\xi)$ is a smooth symplectic manifold $(M_\xi,\omega_\xi)$. 
This symplectic manifold is the  \emph{Marsden-Weinstein
reduced symplectic manifold} for the value $\xi$ of the momentum map $J^R$. Since $H$ is constant on each orbit of the action
of $G_\xi$ on $(J^R)^{-1}(\xi)$, there exists on $M_\xi$ a unique smooth function $H_\xi$ such that $H_\xi\circ p_\xi$ is equal to the restriction of 
$H$ to $(J^R)^{-1}(\xi)$. The restriction to $(J^R)^{-1}(\xi)$ of the Hamiltonian vector field ${\mathcal X}_H$ on $T^*Q$
projects, by the canonical projection $p_\xi:(J^R)^{-1}(\xi)\mapsto M_\xi$, onto the Hamiltonian vector field
${\mathcal X}_{H_\xi}$ on the symplectic manifold $(M_\xi,\omega_\xi)$. Therefore, the determination of solutions of 
the Hamilton equation $\hbox{\bf(H)}$ (\emph{i.e.}, of integral curves of
${\mathcal X}_H$) contained in $(J^R)^{-1}(\xi)$ can be made in two steps. In the first step, one determines 
the integral curves of the Hamiltonian vector field ${\mathcal X}_{H_\xi}$ on the reduced symplectic 
manifold $(M_\xi,\Omega_\xi)$. In the second step (sometimes called \emph{reconstruction}) one determines the integral 
curves of ${\mathcal X}_H$ contained in $(J^R)^{-1}(\xi)$ themselves.       
\par\smallskip

Under the same assumptions, in the Euler-Poincar\'e reduction procedure, one uses the existence of a smooth function
$h:{\frak g}^*\to{\mathbb R}$ such that $H=h\circ J^L$ and the fact that $J^L$ is a Poisson map.  Each
solution of the Hamilton equation $\hbox{\bf(H)}$ 
is mapped by $J^L$ onto a solution of the Euler-Poincar\'e equation $\hbox{\bf(E-P5)}$ (\emph{i.e.}, 
onto an integral curve of the the Hamiltonian vector field ${\mathcal X}_h$ on the Poisson manifold ${\frak g}^*$).
Therefore, the determination of solutions of the Hamilton equation $\hbox{\bf(H)}$ can be made in two steps. 
In the first step, one determines their projection by $J^L$, which are the integral
curves of the Hamiltonian vector field ${\mathcal X}_h$ on the Poisson manigold ${\frak g}^*$. 
This determination can be made easier if one uses the fact that each integral curve of ${\mathcal X}_{H_\xi}$
is contained in a coadjoint orbit (a consequence of the fact that $J^R$ is a first integral). 
In the second step one determines
the integral curves of the compatibility condition $\hbox{\bf (CC)}$, from which the solutions of $\hbox{\bf (H)}$ are easily deduced. 
\par\smallskip 

The Proposition \ref{stat4} and its Corollary \ref{stat5} clearly show that under the assumptions made in this section, 
\emph{i.e.} when the configuration space of our system is a connected Lie group $G$ and when the lift to $T^*G$ of 
the action of $G$ on itself by translations on the right leaves the Hamiltonian invariant, the first steps
of the Marsden-Weinstein and Euler-Poincar\'e reduction procedures are equivalent. In the Marsden-Weinstein reduction procedure, 
one has to determine the integral curves of ${\mathcal X}_{H_\xi}$ on the Marsden-Weinstein 
reduced symplectic manifold $(M_\xi,\omega_\xi)$. In the Euler-Poincar\'e reduction procedure, one has to 
determine the integral curves of ${\mathcal X}_h$ on the symplectic leaf $\Ad^*_{G}\xi$ of the Poisson manifold
${\frak g}^*$. By Corollary \ref{stat5}, $J^L_\xi$ is a symplectic diffeomorphism between these two symplectic manifolds such that $h\circ J^L_\xi=H_\xi$.

\begin{rmks}{\rm \hfill
\par
\noindent
{\bf 1.}\quad Similar results hold, \emph{mutatis mutandis}, when it is the lift to $T^*G$ of the action of $G$ on itself
by translations on the left (instead of on the right) which leaves the Hamiltonian invariant.
\par\smallskip\noindent
{\bf 2.}\quad Let us identify ${\frak g}^*$ with $T^*_eG$, and consider the two momentum maps
$J^R:T^*G\to{\frak g}^*$ and $J^L:T^*G\to{\frak g}^*$, associated to the actions on the right 
$\widehat\Phi^R:T^*G\times G\to T^*G$
and on the left $\Phi^L:G\times T^*G\to T^*G$, respectively. 
The cotangent bundle $T^*G$ being endowed with the Poisson structure associated to its canonical symplectic 
$2$-form $\omega_{T^*G}$ (which is the exterior differential of the Liouville $1$-form), we can define on 
${\frak g}^*$ the Poisson structure for which
$J^R$ is a Poisson map, and the Poisson structure for which $J^L$ is a Poisson map. Each of these structure
is the opposite of the other; however, they are isomorphic by the vector space automorphism
of ${\frak g}^*$ $X\mapsto -X$. 
The formula for the bracket of two functions $f$ and $g$ defined on ${\frak g}^*$ is the same fore these two Poisson structures,
 $$\{f,g\}(\xi)=\Bigl\langle \xi, \bigr[\d f(\xi),\d g(\xi)\bigr]\Bigr\rangle$$
where, in the right hand side, the differentials at $\xi$, $\d f(\xi)$ and $\d g(\xi)$, 
of the functions $f$ and $g$, 
which are linear forms on ${\frak g}^*$, are considered as elements of ${\frak g}$, identified
with $T_eG$. The bracket $\bigl[df(\xi),dg(\xi)\bigr]$ which appears in the right hand side is the bracket 
of fundamental vector fields on $G$ for the action of $G$ onto itself whose lift to $T^*G$ is the action whose momentum map
is the momentum map under consideration. In other words, it is the bracket of vector fields
\begin{itemize}

\item{} invariant by translations \emph{on the left} for the Poisson structure on ${\frak g}^*$ for which $J^R$ is a Poisson map,

\item{} invariant by translations \emph{on the right} for the Poisson structure on ${\frak g}^*$ for which $J^L$ is a Poisson
map.

\end{itemize}

When $G$ acts on $T^*G$ by the action $\widehat\Phi^L$, its action on ${\frak g}^*$ which renders $J^L$ equivariant
is an action \emph{on the left}, whose expression is
 $$(g,\xi)\mapsto \Ad^*_g(\xi)\,,\quad g\in G\,,\quad \xi\in{\frak g}^*\,,$$
and when $G$ acts on $T^*G$ by the action $\widehat\Phi^R$, its action on ${\frak g}^*$ which renders $J^R$ equivariant
is an action \emph{on the right}, whose expression is
 $$(\xi,g)\mapsto \Ad^*_{g^{-1}}(\xi)\,,\quad \xi\in{\frak g}^*\,,\quad g\in G\,.$$
For the adjoint representation our sign convention is the usual one,
 $$\Ad_g(X)=TL_g\circ TR_{{g^{-1}}}(X)\,,\quad g\in G\,,\quad X\in {\frak g}\equiv T_eG\,,$$
and for the coadjoint representation it is
 $$\Ad_g^*(\xi)=(\Ad_{g^{-1}})^t(\xi)\,,\quad g\in G\,,\quad \xi\in{\frak g}^*\equiv T^*_eG\,,$$
where $(\Ad_{g^{-1}})^t:{\frak g}^*\to{\frak g}^*$ is
the transpose of $\Ad_{g^{-1}}:{\frak g}\to{\frak g}$. With these sign conventions
$$\frac{\d (\Ad_{\exp(tX)}Y)}{\d t}\Bigm|_{t=0}=[X,Y]\,,\quad X\ \hbox{and}\ Y\in{\frak g}\,,$$
the bracket $[X,Y]$ in the right hand side being that of vector fields on $G$ invariant by translations \emph{on the left},
which is the most frequently made convention for the bracket on the Lie algebra of a Lie group.
\par\smallskip\noindent

{\bf 3.}\quad The formulae given in this section for the Hamiltonian actions of a Lie group $G$ on 
its cotangent bundle can be generalized, the canonical symplectic form on $T^*G$ being modified by
addition of the pull-back of a closed $2$-form on $G$. This generalization is useful for dealing 
with mechanical systems involving magnetic forces. See for example 
\cite{guilleminsternberg}, \cite{lima} Chapter IV section 4 and
\cite{marle}.
\par\smallskip\noindent

{\bf 4.}\quad Alan Weintein and his students \cite{weinstein, xu} have developed a very nice theory of
\emph{symplectic groupoids} in which the properties of the source and target maps generalize those of the momentum
maps $J^R$ and $J^L$ of the actions of a Lie group on its cotangent bundle. The cotangent bundle of a Lie group
is one of the simplest nontrivial examples of symplectic groupoids, a fact which should convince the reader that
symplectic groupoids are very natural structures rather than artificial mathematical artefacts.
}
\end{rmks}

\section{Symmetry Breaking and Appearance of Semi-direct Products.}\label{SymBreak}
In \cite{holm} the authors write
\emph{\lq\lq It turns out that semidirect products occur under rather general circumstances
when the symmetry in $T^*G$ is broken\rq\rq}. Let us propose an explanation of this remarkable fact.
\par\smallskip

In this section $G$ is a connected $n$-dimensional Lie group and $G_1$ is a closed, connected 
$k$-di\-men\-sio\-nal subgroup of $G$.
The notations $\Phi^R$ and $\Phi^L$ for the actions of $G$ on itself by translations on the right and on the left,
$\widehat\Phi^R$ and $\widehat\Phi^L$ for their lifts to the cotangent bundle $T^*G$, $J^R$ and $J^L$
for their momentum maps, are the same as in Section~\ref{confspaceliegroup}. We assume that $H:T^*G\to{\mathbb R}$ 
is a smooth Hamiltonian invariant by $\widehat\Phi^R_1=\widehat\Phi^R|_{G_1}$, the \emph{restriction to $G_1$ of
the action $\widehat\Phi^R$}, rather than by the action $\widehat\Phi^R$ of
the whole Lie group $G$. 
The Hamiltonian $H$ therefore cannot be written as the composition of the momentum map $J^L$ with a smooth function
defined on ${\frak g}^*$, so the Euler-Poincar\'e equation $\hbox{\bf(E-P4)}$, 
written for the action $\widehat\Phi^L$, cannot be considered as an 
autonomous differential equation on ${\frak g}^*$ for the parametrized curve
$\xi=J^L\circ\zeta$. However, we will prove that under some additional assumptions the action $\widehat\Phi^L$ 
can be extended into a Hamiltonian action of a semi-direct product of $G$ with a finite dimensional 
vector space of smooth functions defined on $G/G_1$, in such a way that the orbits of this extended 
action are the leaves of the foliation of $T^*G$ determined by the symplectic orthogonal of the sub-bundle 
tangent to the orbits of $\widehat \Phi^R_1$. The level sets of the
momentum map of this extended action are the orbits of $\widehat\Phi^R_1$, which will allow us to write the 
Euler-Poincar\'e equation for this extended action instead of for the action $\widehat\Phi^L$.
\par\smallskip
 
\subsection{Two orthogonal foliations of the cotangent bundle}

\begin{lemma}\label{stat6}
The action $\widehat\Phi^R_1$ is Hamiltonian and has $J^R_1=p_{{\frak g}^*_1}\circ J^R$ as momentum map,
where the projection $p_{{\frak g}^*_1}:{\frak g}^*\to{\frak g}^*_1$ is the transpose of the
canonical linear inclusion $i_{{\frak g}_1}:{\frak g}_1\to{\frak g}$. 
The orbits of that action are the intersections of the orbits of $\widehat\Phi^R$ with
the pull-backs $\pi_G^{-1}(gG_1)$, by the canonicat projection $\pi_G:T^*G\to G$, 
of orbits of the action of $G_1$ on $G$ by translations on  the right.
The set $\mathcal F$ of  vectors tangent to these orbits  and its symplectic orthogonal 
$\orth{\mathcal F}$ are both completely integrable vector sub-bundles of $T(T^*G)$, of ranks $k$ and $2n-k$, 
respectively. 
\end{lemma}

\begin{proof}
The fundamental vector fields on 
$T^*G$ for the action $\widehat\Phi^R_1$ are the Hamiltonian vector fields
whose Hamiltonians can be written 
 $$\bigl\langle J^R, i_{{\frak g}_1}(X)\bigr\rangle
  =\bigl\langle (i_{{\frak g}_1})^t\circ J^R,X\bigr\rangle
  =\bigl\langle p_{{\frak g}^*_1}\circ J^R,X\bigr\rangle\,,\quad\hbox{with}\ X\in{\frak g}_1\,.
 $$
Therefore the action $\widehat\Phi^R_1$ is Hamiltonian and has $p_{{\frak g}^*_1}\circ J^R$ as momentum map.
This action being the restriction to $G_1$ of $\widehat\Phi^R:T^*G\times G\to T^*G$,
which projects onto the action $\Phi^R:G\times G\to G$, its orbits are the intersections of the orbits of $\widehat\Phi^R$ with
the pull-backs by $\pi_G$ of cosets $gG_1$, wit $g\in G$. 
Since all these orbits are of the same dimension $k$, the set $\mathcal F$ is a completely integrable 
vector sub-bundle of $T(T^*G)$. Its symplectic orthogonal $\orth{\mathcal F}$ is therefore a rank $2n-k$ 
vector sub-bundle of $T(T^*G)$. This vector sub-bundle is generated by Hamiltonian vector fields   
whose Hamiltonians $f$ are smooth functions on $T^*G$ whose restrictions to each orbit of $\widehat\Phi^R_1$ are constants.
Let $f_1$ and $f_2$ be two such functions. The bracket of the Hamiltonian vector field ${\mathcal X}_{f_1}$ and
${\mathcal X}_{f_2}$ is the Hamiltonian vector field ${\mathcal X}_{\{f_ 1,f_2\}}$. Let $h:T^*G\to {\mathbb R}$ be 
the smooth function 
 $$h=\bigl\langle p_{{\frak g}^*_1}\circ J^R,X\bigr\rangle\,,
 $$
where $X$ is any element in ${\frak g}_1$. Using the Jacobi identity, we can write
 $$i({\mathcal X}_h)\d\bigl(\{f_1,f_2\}\bigr)=\bigl\{h,\{f_1,f_2\}\bigr\}
   =\bigl\{\{h,f_1\},f_2\bigr\}+\bigl\{f_1,\{h,f_2\}\bigr\}=0\,,
 $$
since $\{h,f_1\}=i({\mathcal X}_h)\d f_1=0$ and $\{h,f_2\}=i({\mathcal X}_h)\d f_2=0$, 
the vector field ${\mathcal X}_h$ being tangent to the orbits of $\widehat\Phi^R_1$ 
and the retrictions of the functions $f_1$ and
$f_2$ to each orbit of this action being constants. Since 
${\mathcal F}$ is generated by Hamiltonian vector fields such as ${\mathcal X}_h$,
the restriction of $\{f_1,f_2\}$ to each orbit of $\widehat\Phi^R_1$ is constant. The Frobenius theorem then proves that $\orth{\mathcal F}$ is completely integrable.
\end{proof}

\begin{rmks}{\rm\hfill
\par\noindent
{\bf 1.}\quad Lemma~\ref{stat6} may be seen as a special case of a result due to P.~Libermann
(see \cite{libermann1, libermann2} or \cite{lima} Chapter III Proposition 9.7).

\par\smallskip\noindent
{\bf 2.}\quad On the symplectic manifold $(T^*G,\omega_{T^*G})$ each of the two foliations
$\mathcal F$ and $\orth{\mathcal F}$ is the symplectic orthogonal of the other, and is such that the space of smooth functions whose restrictions to the leaves are constants is closed with respect to the Poisson bracket. If the set of leaves of one of these foliations has a smooth manifold structure for which the canonical projection of $T^*G$ onto this set is a submersion, there exists on this set a unique Poisson structure for which the canonical projection is a Poisson map. The pair of Poisson manifolds made by the sets of leaves when this occurs for both foliations is said to be a \emph{dual pair}, in the terminology introduced by Alan 
Weinstein~\cite{weinsteinlocstruct}. 
}
\end{rmks}

The next two two propositions will allow us to prove that $\mathcal F$ and $\orth{\mathcal F}$ determine indeed a dual pair.

\begin{proposition}\label{stat7}
The leaves of the foliation of $T^*G$ determined by $\orth{\mathcal F}$ are the left invariant affine sub-bundles
whose fibres over the neutral element are affine subspaces of ${\frak g}^*$ whose associated vector subspace is
the annihilator ${\frak g}_1^0$ of the sub-algebra ${\frak g}_1$. Moreover, they coincide with the level sets of the momentum map $J_1^R=p_{{\frak g}^*}\circ J^R:T^*G\to {\frak g}_1^*$. The map which associates 
to each leaf the value taken by $J^R_1$ on that leaf is a bijection of $\leaves(\orth{\mathcal F})$
onto ${\frak g}^*_1$. 
\end{proposition}

\begin{proof}
The maps $J^R:T^*G\to{\frak g}^*$ and 
$p_{{\frak g}^*_1}:{\frak g}^*\to{\frak g}^*_1$ both are surjective submersions. Therefore $J^R_1=p_{{\frak g}^*_1}\circ J^R$ is a surjective submersion. Let
$\zeta_1$ and $\zeta_2$ be two elements of $T^*G$. Using the expression of $J^R$, we obtain, for any $X\in{\frak g}_1$,
 $$\bigl\langle J^R_1(\zeta_1)-J^R_1(\zeta_2),X\bigr\rangle  
  =\bigl\langle (TL_{\pi_G(\zeta_1)})^t(\zeta_1)-TL_{\pi_G(\zeta_2)})^t(\zeta_2),X
    \bigr\rangle\,.
 $$
Therefore $J^R_1(\zeta_1)=J^R_1(\zeta_2)$ if and only if
$(TL_{\pi_G(\zeta_1)})^t(\zeta_1)-TL_{\pi_G(\zeta_2)})^t(\zeta_2)\in{\frak g}_1^0$, 
the annihilator of ${\frak g}_1$. The level sets of $J_1^R$ are therefore the
left invariant affine sub-bundles whose fibres over the neutral element are affine subspaces 
of ${\frak g}^*$ whose associated vector subspace is ${\frak g}_1^0$. Since for each $\zeta\in T^*G$ $\ker T_\zeta J_1^R$ is the symplectic orthogonal of ${\mathcal F}_\zeta$, the leaves of the foliation of $T^*G$ determined by $\orth{\mathcal F}$ are the connected components of the level sets of $J_1^R$. But 
since $G$ is assumed to be connected, these level sets are connected, therefore coincide with the elements
of $\leaves(\orth{\mathcal F})$. The last assertion immediately follows.
\end{proof} 

\begin{proposition}\label{stat8}
Let $\varpi:G\to G/G_1$ be the canonical projection which associates to each $g\in G$ the coset $gG_1$. The map 
$(J^L,\varpi\circ\pi_G):T^*G\to{\frak g}^*\times(G/G_1)$
is a surjective submersion, whose restriction to 
each leaf of the foliation determined by 
${\mathcal F}$ is constant. The map defined on the set 
$\leaves({\mathcal F})$ of leaves of that foliation, which associates to each leaf the value taken by $(J^L,\varpi\circ\pi_G)$ on that leaf, is a bijection of $\leaves({\mathcal F})$ onto
${\frak g}^*\times (G/G_1)$. 
\end{proposition}

\begin{proof}
The maps $J^L:T^*G\to{\frak g}^*$, $\pi_G:T^*G\to G$ and $\varpi:G\to G/G_1$ are surjective submersions. Therefore $\varpi\circ\pi_G$ and $(J^L,\varpi\circ\pi_G)$ are submersions, and its expression proves that  
$(J^L,\varpi\circ\pi_G)$ is surjective. We already know (Lemma \ref{stat1}) that the level sets of $J^L$ 
are the orbits of
the action $\widehat\Phi^R$. We have seen~(Lemma \ref{stat6}) that the orbits of 
$\widehat\Phi^R_1$, in other words the leaves of the foliation determined by 
${\mathcal F}$, are the intersections of the level sets of $J^L$ with
the pull-backs by the canonical projection $\pi_G:T^*G\to G$, 
of orbits of the action of $G_1$ on $G$ by translations on  the right. Since these orbits are inverse images of points in $G/G_1$ by the projection $\varpi$, each leaf of the foliation determined by ${\mathcal F}$ is a level set of the map 
$(J^L,\varpi\circ\pi_G)$. Therefore, $(J^L,\varpi\circ\pi_G)$ determines indeed a bijection of $\leaves({\mathcal F})$ onto ${\frak g}^*\times (G/G_1)$.        
\end{proof}

\begin{proposition}\label{stat9}
On each of the two smooth manifolds ${\frak g}^*_1$ and ${\frak g}^*\times(G/G_1)$, there exists a unique 
Poisson structure for which, when $T^*G$ is equipped with the Poisson structure associated to its symplectic 
form $\omega_{T^*G}$, the maps $J^R_1:T^*G\to{\frak g}^*_1$ and 
$(J^L,\varpi\circ\pi_G):T^*G\to{\frak g}^*\times(G/G_1)$ are Poisson maps.
Moreover, there exists a unique smooth function $h:{\frak g}^*\times(G/G_1)\to{\mathbb R}$ such that
 $$H=h\circ (J^L,\varpi\circ\pi_G)$$
and the Hamiltonian vector field ${\mathcal X}_H$ on the symplectic manifold $(T^*G,\omega_{T^*G})$ is mapped, by
the prolongation to vectors of the submersion $(J^L,\varpi\circ\pi_G)$, onto the Hamiltonian vector field
${\mathcal X}_h$ on the Poisson manifold ${\frak g}^*\times(G/G_1)$.
\end{proposition}

\begin{proof} The first assertion follows from Propositions \ref{stat7} and \ref{stat8} which show that the 
pair of manifolds $\bigl({\frak g}_1^*,{\frak g}^*\times(G/G_1)\bigr)$ is a dual pair of Poisson manifolds 
in the sense of Alan Weinstein. The second assertion is an immediate consequence of the constancy of $H$ on each leaf of the foliation determined by $\mathcal F$ and of the fact that $(J^L,\varpi\circ\pi_G)$ is a submersion.
\end{proof}

\begin{rmk}{\rm
Proposition \ref{stat9} shows that 
as a first step for the determination of integral curves of the Hamiltonian vector field ${\mathcal X}_H$ on $T^*G$, one can determine their projections by $(J^l,\varpi\circ\pi_G)$ on the Poisson manifold
${\frak g}^*\times(G/G_1)$, which are the integral curves of ${\mathcal X}_h$. Although it is not the dual space of a finite dimensional Lie algebra, that Poisson manifold  can be used for reduction.
}
\end{rmk}

\subsection{The extended action of a semi-direct product}
\par\smallskip

Instead of the Poisson manifold ${\frak g}^*\times(G/G_1)$, one may use for reduction the dual
space of the semi-direct product of $\frak g$ with a finite dimensional vector space. The differential equation one has to solve in a first step will be defined on a vector space instead of on the product of the vector space ${\frak g}^*$ with the homogeneous space $G/G_1$, which may appear as an advantage; however, the dimension of that vector space will generally be larger than the dimension of ${\frak g}^*\times(G/G_1)$. 

The next Lemma identifies the Hamiltonian vector fields which generate the vector sub-bundle $\orth{\mathcal F}$.

\begin{lemma}\label{stat10}
For each $\zeta\in T^*G$, the fibre of $\orth{\mathcal F}$ over $\zeta$ is the direct sum of the two vector subspaces of $T_\zeta(T^*G)$:
the tangent space at $\zeta$ to the $\widehat\Phi^L$-orbit of that point,
and the vector subspace made by the values at $\zeta$ of the Hamiltonian vector fields on $T^*G$ whose Hamiltonian can be written as $h\circ\varpi\circ\pi_G$, where $h:G/G_1\to{\mathbb R}$ is a smooth function.
\end{lemma}

\begin{proof}
Since $\mathcal F$ and $\orth{\mathcal F}$ are two symplectically complete (in the sense of P.~Libermann
\cite{libermann1, libermann2}) and symplectically orthogonal sub-bundles of $T(T^*G)$, each of them is the set of values of Hamiltonian vector fields whose Hamiltonians are functions constant on the leaves of the foliation determined by the other one. Therefore $\orth{\mathcal F}$ is generated by the values of Hamiltonian vector fields whose Hamiltonians are composed of the map $(J^L,\varpi\circ\pi_G)$ with a smooth function defined on 
${\frak g}^*\times(G/G_1)$. For each $\zeta\in T^*G$, the tangent space at $\zeta$ to the $\widehat\Phi^L$-orbit of that point and the vertical tangent space (kernel of $T_\zeta\pi_G$) are two complementary vector subspaces of $T(T^*G)$. The announced result follows from the facts that Hamiltonian vector fields whose Hamiltonians are composed of $J^L$ with a function defined on ${\frak g}^*$ are tangent to the $\widehat\Phi^L$-orbits, while
Hamiltonian vector fields whose Hamiltonians are composed of $\varpi\circ\pi_G$ with a function defined on $G/G_1$ are vertical.
\end{proof}

\begin{lemma}\label{stat11}
Let $f\in C^\infty(G/G_1,{\mathbb R})$ be a smooth function. The flow of the Hamiltonian vector field whose Hamiltonian is
$f\circ\varpi\circ\pi_G:T^*G\to {\mathbb R}$ is
 $$\Psi_f(t,\zeta)=\zeta-t\d (f\circ\varpi)\bigl(\pi_G(\zeta)\bigr)\,.$$
\end{lemma}

\begin{proof}
The map $\Psi_f:{\mathbb R}\times T^*G\to T^*G$ is the flow of the vertical vector field $Z_f$ on $T^*G$ whose value at $\zeta\in T^*G$ is $-\d (f\circ\varpi)\bigl(\pi_G(\zeta)\bigr)$ (where the tangent space at $\zeta$ to the fibre $T^*_{\pi_G(\zeta)}G$ is identified with that vector space). Using the expression of the Liouville form 
$\eta$ and the fact that $\omega_{T^*G}=\d \eta$, we can write
 $$i(Z_f)\eta=0\,,\quad {\mathcal L}(Z_f)\eta=i(Z_f)\omega_{T^*G}=-\d(f\circ\varpi\circ\pi_G)\,,$$
where ${\mathcal L}(Z_f)\eta$ is the Lie derivative of $\eta$ with respect to $Z_f$. Therefore $Z_f$ is a Hamiltonian vector field, with $f\circ\varpi\circ\pi_G$ as Hamiltonian. 
\end{proof}

\begin{lemma}\label{stat12}
The map which associates to each $g\in G$ the linear transformation of $C^\infty(G,{\mathbb R})$
 $$f\mapsto L_{g^{-1}}^*(f)=f\circ L_{g^{-1}}$$
is a linear representation of $G$, which maps onto itself the vector subspace of functions whose restrictions to orbits of $\widehat\Phi^R_1$ are constants.
\end{lemma}

\begin{proof}
This is an immediate consequence of the fact that the actions $\widehat\Phi^L$ and $\widehat\Phi^R$ commute.
\end{proof}

\begin{rmk}\label{assumption}{\rm
The vector space $C^\infty(G/G_1,{\mathbb R})$ can be considered as an infinite-di\-men\-sio\-nal Abelian Lie group.
Lemmas \ref{stat11} and \ref{stat12} show that $G\times C^\infty(G/G_1,{\mathbb R})$ can be equipped with the structure of a semi-direct product of groups and that it acts on the symplectic manifold
$(T^*G,\omega_{T^*G})$ by a Hamiltonian action. The map defined on $T^*G$ with values in the product of 
${\frak g}^*$ with the space of distributions on $G/G_1$ (in the sense of Laurent Schwartz, \emph{i.e.}, the 
dual of $C^\infty(G/G_1,{\mathbb R})$)
 $$\zeta\mapsto\bigl(J^L(\zeta),\delta_{\varpi\circ\pi_G(\zeta)}\bigr)\,,$$
where $\delta_{\varpi\circ\pi_G(\zeta)}$ is the Dirac distribution at 
$\varpi\circ\pi_G(\zeta)$, can be considered as a momentum map (in a generalized sense) of that action.
This explains why a symmetry break in $T^*G$ often causes the appearance of semi-direct product of groups.
In the next Proposition, we assume that there exists a finite-dimensional vector subspace $V$ of $C^\infty(G/G_1,{\mathbb R})$ which is mapped onto itself by the linear representation of $G$ described in Lemma \ref{stat12} and separates points (\emph{i.e.}, which is such that for any pair of distinct points in $G/G_1$, there exists at least one function which belongs to that space and takes different values at these points). 
}
\end{rmk}

\begin{proposition}\label{stat13}
We assume that there exists a finite-dimensional vector subspace $V$ of $C^\infty(G/G_1,{\mathbb R})$ 
which is mapped onto itself by the linear representation of $G$ described in Lemma \ref{stat12} 
and separates points (in the sense explained in Remark \ref{assumption}). Then there exists a
Hamiltonian action of the semi-direct product $G\times V$ which extends the action 
$\widehat \Phi^L$ of $G$ on $T^*G$. The momentum map $(J^L,K)$ of that action, which takes its values in
${\frak g}\times V^*$, has as first component the momentum map $J^L$ of the action $\widehat \Phi^L$. Its second
component $K:T^*G\to V^*$ is given by
 $$\bigl\langle K(\zeta),f\bigr\rangle=f\bigl(\varpi\circ\pi_G(\zeta)\bigr)\,,
\quad \zeta\in T^*G\,,\quad f\in V\,.$$
Moreover, $(J^L,K)$ is constant on each orbit of the action $\widehat \Phi_1^R$ and the Hamiltonian $H$ is constant on each level set of $(J^L,K)$. If a smooth function $h:{\frak g}^*\times V^*\to{\mathbb R}$ is such that
$H=h\circ(J^L,K)$, $(J^L,K)$ maps each integral curve of the Hamiltonian vector field ${\mathcal X}_H$ on the symplectic manifold $(T^*G,\omega_{T^*G})$ onto an integral curve of the Hamiltonian vector field 
${\mathcal X}_h$ on the Poisson manifold ${\frak g}\times V^*$.  
\end{proposition}

\begin{proof}
The assumption made shows that the semi-direct product of groups structure which, by Lemma \ref{stat12} and Remark \ref{assumption}, exists on $G\times C^\infty(G/G_1,{\mathbb R})$, as well as its Hamiltonian action on $T^*G$,
yield by restriction a semi-direct product of groups structure on $G\times V$ and a Hamiltonian action of that group on $T^*G$. The expression of the momentum map $(J^L,K)$ follows from that of generalized momentum map
of the action of $C^\infty(M,{\mathbb R})$ given in Remark \ref{assumption}. The other assertions come from the facts that $V$ separates points and that $(J^L,K)$ is a Poisson map. 
\end{proof}

\begin{rmk}{\rm
The map $(J^L,K):T^*G\to {\frak g}\times V^*$ may not be surjective. Therefore, the smooth function
$h:{\frak g}\times V^*\to{\mathbb R}$ such that $H=h\circ(J^L,K)$ may not be unique.
}
\end{rmk}
\par\smallskip\noindent

{\bf Example}\quad
The motion of a rigid body with a fixed point considered in Subsection \ref{rigidbody} is a system which satisfies the assumption of Proposition \ref{stat13}. For each configuration of the rigid body, the center of mass of the body lies on a sphere embedded in the physical space $E$, centered on the fixed point. That sphere realizes a natural embedding of the homogeneous space $G/G_1$ into $E$. The $3$-dimensional vector space $V$ of functions 
on $G/G_1$ is the vector space of linear functions on the physical space $E$ composed with that natural embedding.

\section*{Acknowledgements}
\addcontentsline{toc}{section}{Acknowledgements}
The author thanks his colleagues and friends Maylis Irigoyen, Alain Albouy, Marc Chaperon, Alain Chenciner, Laurent Lazzarini, Claude Vall\'ee and G\'ery de Saxc\'e for helpful discussions. He thanks Alan Weinstein 
for his encouragements and his interest in this work. Although retired he received
the material and moral support of his former Institutions, the \emph{Institut de Math\'ematiques de Jussieu} and the \emph{Universit\'e Pierre et Marie Curie}.

Charles-Michel Marle \\
Institut de Math\'ematiques de Jussieu\\
Universit\'e Pierre et Marie Curie\\
Paris, France\\
{\it E-mail address}: {\tt charles-michel.marle@polytechnique.org}\\[0.3cm]

\label{last}
\end{document}